\documentclass[11pt]{amsart}

\usepackage[T1]{fontenc}
\usepackage[utf8]{inputenc}
\usepackage{lmodern}

\usepackage[american]{babel}
\usepackage[babel, final]{microtype}
\usepackage{amssymb}
\usepackage[all]{xy}
\usepackage{amsmath}

\usepackage{ifthen}
\usepackage{ifpdf}
\usepackage{multirow}
\usepackage[shortlabels]{enumitem}
\makeatletter
\newcommand{\mylabel}[2]{#2\def\@currentlabel{#2}\label{#1}}
\makeatother

\ifpdf
  \usepackage[pdftex, dvipsnames]{xcolor}
  \usepackage[pdftex, final]{graphicx}
  \usepackage{caption}
  \usepackage{subcaption}
  \usepackage{pinlabel}
  \usepackage[pdftex,%
    letterpaper,%
    includehead,%
    includefoot,%
    nomarginpar,%
    lmargin=1in,%
    rmargin=1in,%
    tmargin=1in,%
    bmargin=1in,%
  ]{geometry}
  \usepackage[pdftex,%
    final,%
    colorlinks=true,%
    linkcolor=NavyBlue,%
    citecolor=NavyBlue,%
    filecolor=NavyBlue,%
    menucolor=NavyBlue,%
    urlcolor=NavyBlue,%
    bookmarks=true,%
    bookmarksdepth=3,%
    bookmarksnumbered=true,%
    bookmarksopen=true,%
    bookmarksopenlevel=2,%
  ]{hyperref}
  \hypersetup{
    pdftitle={Double-point enhanced GRID invariants and Lagrangian cobordisms},
    pdfauthor={Ashton Lewis, Zachary Ojakli, Ina Petkova, Benjamin Shapiro},
    pdfsubject={Lagrangian cobordisms},
    pdfkeywords={Lagrangian cobordism, Legendrian knots, Heegaard Floer 
      homology}
  }
\else
  \usepackage[dvips, dvipsnames]{xcolor}
  \usepackage[dvips, final]{graphicx}
  \usepackage{caption}
  \usepackage{subcaption}
  \usepackage{pinlabel}
  \usepackage[dvips,%
    letterpaper,%
    includehead,%
    includefoot,%
    nomarginpar,%
    lmargin=1in,%
    rmargin=1in,%
    tmargin=1in,%
    bmargin=1in,%
  ]{geometry}
  \usepackage[ps2pdf,%
    final,%
    colorlinks=true,%
    linkcolor=NavyBlue,%
    citecolor=NavyBlue,%
    filecolor=NavyBlue,%
    menucolor=NavyBlue,%
    urlcolor=NavyBlue,%
    bookmarks=true,%
    bookmarksdepth=3,%
    bookmarksnumbered=true,%
    bookmarksopen=true,%
    bookmarksopenlevel=2,%
  ]{hyperref}
  \hypersetup{
    pdftitle={Double-point enhanced GRID invariants and Lagrangian cobordisms},
    pdfauthor=Ashton Lewis, Zachary Ojakli, Ina Petkova, Benjamin Shapiro},
    pdfsubject={Lagrangian cobordisms},
    pdfkeywords={Lagrangian cobordism, Legendrian knots, Heegaard Floer 
      homology}
  
\fi

\usepackage{mathdots}
\usepackage{float}
\usepackage[mathscr]{euscript}
\usepackage[normalem]{ulem}
\usepackage{tikz}
\usepackage{tikz-cd}
\usepackage{mathtools}
\usepackage{stmaryrd}
\SetSymbolFont{stmry}{bold}{U}{stmry}{m}{n}
\usetikzlibrary{arrows,automata,fit}
\usetikzlibrary{matrix,arrows}
\usetikzlibrary{decorations.pathreplacing,calligraphy}

\input{macros}

\newcommand{\set}[1]{\mathchoice%
  {\left\lbrace #1 \right\rbrace}%
  {\lbrace #1 \rbrace}%
  {\lbrace #1 \rbrace}%
  {\lbrace #1 \rbrace}%
}
\newcommand{\setc}[2]{\mathchoice%
  {\left\lbrace #1 \, \middle\vert \, #2 \right\rbrace}%
  {\lbrace #1 \, \vert \, #2 \rbrace}%
  {\lbrace #1 \, \vert \, #2 \rbrace}%
  {\lbrace #1 \, \vert \, #2 \rbrace}%
}
\newcommand{\paren}[1]{\mathchoice%
  {\left( #1 \right)}%
  {( #1 )}%
  {( #1 )}%
  {( #1 )}%
}

\newcommand{\abs}[1]{\mathchoice%
  {\left\lvert #1 \right\rvert}%
  {\lvert #1 \rvert}%
  {\lvert #1 \rvert}%
  {\lvert #1 \rvert}%
}

\newcommand{\dirsum}{\oplus}                 
\newcommand{\union}{\cup}                    
\newcommand{\intersect}{\cap}                

\newcommand{\comp}{\circ}                    
\newcommand{\cross}{\times}                  

\newcommand{\isom}{\cong}                    

\newcommand{\numset}[1]{\mathbb{#1}}

\newcommand{\Z}{\numset{Z}}

\newcommand{\R}{\numset{R}}

\DeclareMathOperator{\Id}{Id}

\DeclareMathOperator{\Cone}{Cone}

\newcommand{\bdy}{\partial} 

\newcommand{\xistd}{\xi_{\mathrm{std}}}
\DeclareMathOperator{\tb}{tb}
\DeclareMathOperator{\rot}{r}
\DeclareMathOperator{\self}{sl}

\newcommand{\opminus}[1]{#1^-}

\newcommand{\ophat}[1]{\widehat{#1}}
\newcommand{\optilde}[1]{\widetilde{#1}}

\newcommand{\curves}[1]{\boldsymbol{#1}}
\newcommand{\alphas}[1][]{%
  \ifthenelse{\equal{#1}{}}{\curves{\alpha}}{\curves{\alpha^{#1}}}
}
\newcommand{\betas}[1][]{%
  \ifthenelse{\equal{#1}{}}{\curves{\beta}}{\curves{\beta_{#1}}}
}

\newcommand{\gen}[1]{\mathbf{#1}}

\newcommand{\x}{\gen{x}}
\newcommand{\y}{\gen{y}}
\newcommand{\z}{\gen{z}}

\DeclareMathOperator{\gr}{gr}

\newcommand{\emptypoly}[1]{#1^\circ}

\newcommand{\grid}{\mathbb{G}}

\DeclareMathOperator{\GC}{GC}
\DeclareMathOperator{\GH}{GH}

\newcommand{\GCm}{\opminus{\GC}}

\newcommand{\GCt}{\optilde{\GC}}

\newcommand{\GHm}{\opminus{\GH}}
\newcommand{\GHh}{\ophat{\GH}}
\newcommand{\GHt}{\optilde{\GH}}

\newcommand{\markers}[1]{\mathbb{#1}}
\newcommand{\OO}{\markers{O}}
\newcommand{\XX}{\markers{X}}
\newcommand{\countJ}{\mathcal{J}}

\newcommand{\eRect}{\emptypoly{\Rect}}

\DeclareMathOperator{\Int}{Int}

\newcommand*{\alphastd}{\alpha_{\mathrm{std}}}

\renewcommand*{\tb}{\mathit{tb}}
\renewcommand*{\rot}{\mathit{r}}

\newcommand*{\fGC}{\mathcal{GC}}

\newcommand*{\fGCt}{\optilde{\fGC}}

\newcommand*{\Phih}{\ophat{\Phi}}

\newcommand{\GCmb}{\opminus{\GC}_{\mathit{big}}}
\newcommand{\GHtb}{\optilde{\GH}_{\mathit{big}}}
\newcommand{\GHhb}{\ophat{\GH}_{\mathit{big}}}
\newcommand{\GHmb}{\opminus{\GH}_{\mathit{big}}}

\newcommand{\GChb}{\ophat{\GC}_{\mathit{big}}}
\newcommand{\GCtb}{\optilde{\GC}_{\mathit{big}}}
\newcommand*{\fGCtb}{\fGCt_{\mathit{big}}}

\newcommand*{\fGCtbGp}{\fGCtb\paren{\Gp}}
\newcommand*{\fGCtbGm}{\fGCtb\paren{\Gm}}

\newcommand*{\Legm}{\Leg_-}
\newcommand*{\Legp}{\Leg_+}
\newcommand*{\Lag}{L}

\newcommand*{\lambdah}{\ophat{\lambda}}
\newcommand*{\lambdahp}{\lambdah^+}
\newcommand*{\lambdahm}{\lambdah^-}
\newcommand*{\lambdahpm}{\lambdah^\pm}

\newcommand*{\lambdat}{\optilde{\lambda}}
\newcommand*{\lambdatp}{\lambdat^+}
\newcommand*{\lambdatm}{\lambdat^-}
\newcommand*{\lambdatpm}{\lambdat^\pm}

\newcommand*{\commmapb}{\mathcal{C}_{\mathit{big}}}

\newcommand*{\oswapb}{\mathcal{P}_{O, \mathit{big}}}
\newcommand*{\xswapb}{\mathcal{P}_{X, \mathit{big}}}
\newcommand*{\birthmapb}{\mathcal{B}_{\mathit{big}}}

\newcommand*{\commmapxb}{C_{\mathit{big}}}

\newcommand*{\oswapxb}{P_{O, \mathit{big}}}

\newcommand{\birthemap}{e}
\newcommand{\birthpsimap}{\psi_{big}}
\newcommand{\birthpimap}{\Pi}

\newcommand*{\disjunion}{\sqcup}

\newcommand{\bira}{a}
\newcommand{\birb}{b}

\newcommand*{\AB}{\mathit{AB}}
\newcommand*{\NB}{\mathit{NB}}
\newcommand*{\AN}{\mathit{AN}}
\newcommand*{\NN}{\mathit{NN}}

\newcommand{\ABt}{\widetilde{\AB}}
\newcommand{\ANt}{\widetilde{\AN}}
\newcommand{\NBt}{\widetilde{\NB}}
\newcommand{\NNt}{\widetilde{\NN}}

\newcommand{\J}{\mathcal{J}}
\newcommand{\I}{\mathcal{I}}

\newcommand{\X}{\mathbb{X}}
\newcommand{\G}{\mathbb{G}}
\newcommand{\Gp}{\mathbb{G}_+}
\newcommand{\Gm}{\mathbb{G}_-}

\DeclareMathOperator{\grGH}{\widetilde{GH}}
\newcommand{\SG}{S(\G)}
\newcommand{\SGp}{S(\Gp)}
\newcommand{\SGm}{S(\Gm)}

\newcommand{\Leg}{\Lambda}
\newcommand{\Transv}{\mathrm{T}}

\newcommand{\xp}{\mathbf{x}^+}
\newcommand{\xm}{\mathbf{x}^-}
\newcommand{\xpm}{\mathbf{x}^\pm}
\newcommand{\lp}{\lambda^+}
\newcommand{\lm}{\lambda^-}

\newcommand{\np}{n^+}

\newcommand{\npb}{n^+_{\mathit{big}}}
\newcommand{\nmb}{n^-_{\mathit{big}}}

\newcommand{\std}{\mathrm{std}}

\newcommand{\bdt}{\widetilde{\partial}_{\OO}}
\newcommand{\bdtb}{\widetilde{\partial}_{\OO, \mathit{big}}}
\newcommand{\bdtox}{\widetilde{\partial}_{\OO, \XX}}
\newcommand{\bdtoxb}{\widetilde{\partial}_{\OO, \XX, \mathit{big}}}
\newcommand{\bdm}{\partial_{\XX}^-}
\newcommand{\bdmb}{\partial_{\XX, \mathit{big}}^-}

\newcommand{\dtab}{\widetilde{\delta}_\AB}
\DeclareMathOperator{\Rect}{Rect}
\DeclareMathOperator{\Tri}{Tri}
\DeclareMathOperator{\Pent}{Pent}

\newcommand{\hxnb}{\mathcal H_{X_2, \mathit{big}}^N}
\newcommand{\hxib}{\mathcal H_{X_2, \mathit{big}}^I}
\newcommand{\homb}{\mathcal H_{O_1, \mathit{big}}^-}
\newcommand{\hoxmb}{\mathcal H_{O_1, X_2, \mathit{big}}^-}

\newcommand{\RectGp}[2]{\Rect_{\Gp}(#1,#2)}

\newcommand{\RectAB}[2]{\Rect_{\AB}(#1,#2)}

\newcommand{\ABcomplex}{\paren{\ABt,\dtab}}

\newcommand{\lpb}{\lambda^+_{\mathit{big}}}
\newcommand{\lmb}{\lambda^-_{\mathit{big}}}
\newcommand{\lpmb}{\lambda^\pm_{\mathit{big}}}
\newcommand{\thetab}{\theta_{\mathit{big}}}
\newcommand{\thetahb}{\ophat{\theta}_{\mathit{big}}}

\newcommand*{\lhpb}{\lambdahp_{\mathit{big}}}
\newcommand*{\lhmb}{\lambdahm_{\mathit{big}}}
\newcommand*{\lhpmb}{\lambdahpm_{\mathit{big}}}

\newcommand*{\ltpb}{\lambdatp_{\mathit{big}}}
\newcommand*{\ltmb}{\lambdatm_{\mathit{big}}}
\newcommand*{\ltpmb}{\lambdatpm_{\mathit{big}}}

\newcommand {\al}{\alpha}

\newcommand {\be}{\beta}

\newcommand*{\cref}{\fullfref}

\newcommand{\xpr}{\x'}

\newcommand{\Pentx}{\mathrm{Pent}_\XX(\x,\y')}
\newcommand{\Pento}{\mathrm{Pent}_\OO(\x,\y')}

\newcommand{\Pentxl}{\mathrm{Pent}^\ell_\XX(\x,\y')}
\newcommand{\Pentol}{\mathrm{Pent}^\ell_\OO(\x,\y')}
\newcommand{\Pentl}{\mathrm{Pent}^\ell(\x,\y')}

\title{Double-point enhanced GRID invariants and Lagrangian cobordisms}

\author[A. Lewis]{Ashton Lewis}
\address{Department of Mathematics \\ Dartmouth College \\ Hanover, NH 03755}
\email{\href{mailto:ashton.r.lewis.25@dartmouth.edu}{ashton.r.lewis.25@dartmouth.edu}}

\author[Z. Ojakli]{Zachary Ojakli}
\address{Department of Mathematics \\ Dartmouth College \\ Hanover, NH 03755}
\email{\href{mailto:zachary.d.ojakli.25@dartmouth.edu}{zachary.d.ojakli.25@dartmouth.edu}}

\author[I. Petkova]{Ina Petkova}
\address{Department of Mathematics \\ Dartmouth College \\ Hanover, NH 03755}
\email{\href{mailto:ina.petkova@dartmouth.edu}{ina.petkova@dartmouth.edu}}
\urladdr{\url{https://math.dartmouth.edu/~ina/}}

\author[B. Shapiro]{Benjamin Shapiro}
\address{Department of Mathematics \\ Dartmouth College \\ Hanover, NH 03755}
\email{\href{mailto:benjamin.i.shapiro.gr@dartmouth.edu}{benjamin.i.shapiro.gr@dartmouth.edu}}

\begin{document}

\begin{abstract}
We define an invariant of Legendrian links in the double-point enhanced grid homology of a link, and prove that it obstructs decomposable Lagrangian cobordisms in the symplectization 
of the standard contact structure on $\R^3$.
\end{abstract}

\maketitle

\section{Introduction}\label{sec:intro}
 Two central and difficult problems in low-dimensional contact and symplectic geometry are to classify Legendrian and transverse links in a given contact manifold, as well as to understand whether there exists a Lagrangian cobordism from one Legendrian link to another. 
 This is challenging even in the simplest case where the underlying contact manifold is $\mathbb R^3$ with the standard contact structure
  \[
  \xi_{\std}=\ker(\alpha_{\std}),\quad \alpha_{\std}=dz-y\,dx
\] 
 and cobordisms are studied in the symplectization 
 \[(\R_t \cross \R^3, d (e^t \alphastd))\] 
 of $(\mathbb R^3, \xi_{\std})$. The classical invariants---the Thurston-Bennequin and rotation numbers of a Legendrian link and the self-linking number of a transverse link---do not provide a complete answer to either question. Thus, one is always on the hunt for new \emph{effective} invariants; that is, invariants that can be used to distinguish links with the same classical invariants, or to obstruct the existence of a cobordism when the classical invariants cannot. An example of such effective invariants are the so-called \emph{GRID invariants} in knot Floer homology. In the present article, we study an analogue of the GRID invariants in a slight variation of knot Floer homology, known as \emph{double-point enhanced knot Floer homology}, due to Lipshitz \cite{L06}.

Knot Floer homology is a powerful invariant of null-homologous knots and links in closed, oriented $3$-manifolds. It comes in various flavors, and generally takes the form of a chain complex, well-defined up to homotopy equivalence. It was originally defined using Heegaard diagrams and pseudo-holomorphic disks \cite{hfk, Ras05}. For links in $S^3$, there is a combinatorial description of knot Floer homology, called \emph{grid homology}, defined using a grid diagram presentation of the link \cite{MOS09, MOST07, OSS15}. The generators of the \emph{grid chain complex} are certain tuples of points on the grid, and the differential counts certain rectangles on the grid interpolating between pairs of generators. 

Given a grid $\grid$ that represents a Legendrian link $\Leg\subset (\mathbb R^3, \xi_{\std})$, Ozsv\'ath, Szab\'o, and Thurston  \cite{OST08} associate to it two \emph{canonical generators} $\xp(\grid)$ and $\xm(\grid)$ which are cycles in the grid chain complex, and show that the corresponding homology classes are invariants of the Legendrian link; those classes are denoted $\lp(\Leg)$ and $\lm(\Leg)$. From $\xp(\grid)$ one similarly gets an invariant $\theta(T)$ of the transverse link $ T \subset(\R^3, \xi_{\std})$ that is the positive transverse pushoff of $\Leg$. These invariants have been used with great success to distinguish pairs of Legendrian knots, as well as pairs of transverse knots, when the classical invariants cannot; see for example \cite{OST08, NOT08, CN13atlas}. They have also been used to obstruct the existence of exact decomposable Lagrangian cobordisms between Legendrian knots \cite{BLW22, JPSWW22}.

In \cite{L06}, Lipshitz proposed and studied a different version of knot Floer homology in which the differential counts more pseudo-holomorphic curves than the traditionally-studied differential; for an expanded discussion which also focuses on the combinatorial case, see \cite{L09}. In the grid-diagram formulation, this variation drops one of the conditions on the rectangles counted by the differential, namely that they be \emph{empty}; this formulation is described in \cite[Section 5.5]{OSS15} and referred to as \emph{double-point enhanced grid homology}. It is an open question whether the double-point enhancement gives more information than grid homology \cite[Problem 17.2.5]{OSS15}. It satisfies a skein relation just like grid homology, and for knots with 17 or fewer crossings and for quasi-alternating knots it can be recovered from grid homology \cite{T23}, but little else is known about the similarities and differences between the two theories. 

In this paper, we consider the canonical generators $\xp(\grid)$ and $\xm(\grid)$ and show that they give rise to effective Legendrian and transverse link invariants in double-point enhanced grid homology. We also show that these \emph{double-point enhanced GRID invariants} can effectively obstruct decomposable Lagrangian cobordisms. Before we state our results, we set up some basic notation, following \cite{OSS15}. 

The simplest version of double-point enhanced grid homology is the \emph{tilde} version $\GHtb$, which is the homology of a chain complex $\GCtb(\grid)$ defined over the ground ring $\mathbb F[v]$, where $\mathbb F = \Z/2\Z$ is the field with two elements; it corresponds to the tilde version of grid homology $\GHt$. Another version, which contains more information, is the \emph{minus} version $\GHmb$; for knots, it can be thought of as a module over $\mathbb F[U,v]$, and corresponds to $\GHm$. 

\subsection{The double-point enhanced GRID invariants}
\label{ssec:intro-invts}
Suppose $\grid$ represents a Legendrian link $\Leg\subset (\mathbb R^3, \xi_{\std})$, and let $\lpb(\grid)$ and $\lmb(\grid)$ denote the homology classes of $\xp(\grid)$ and $\xm(\grid)$ in $\GHmb(\grid)$. We prove that these homology classes are invariants of the Legendrian link type of $\Leg$:
\begin{theorem}
  \label{thm:legm}
  Suppose that $\grid$ and $\grid'$ are two grid diagrams that represent the same Legendrian link $\Leg \subset (S^3, \xistd)$. Then, there exists a bigraded isomorphism 
  \[ \Phi \colon \GHmb (\grid) \to \GHmb (\grid')\]
  with $\Phi(\lpb(\grid)) = \lpb(\grid')$ and $\Phi(\lmb(\grid)) = \lmb(\grid')$.
\end{theorem}
This justifies writing $\lpmb(\Leg)$ and calling these homology classes \emph{double-point enhanced Legendrian GRID invariants}.

Considering instead the quotient complex $\GChb(\grid) = \GCmb(\grid)/U$, one gets the \emph{simply blocked double-point enhanced grid homology}  $\GHhb(\grid)$. The grid states $\xpm(\grid)$ can be viewed as cycles in $\GChb(\grid)$. The corresponding homology classes $\lhpmb(\grid)\in \GHhb(\grid)$ are also Legendrian link invariants:
\begin{corollary}
  \label{thm:legh}
  Suppose that $\grid$ and $\grid'$ are two grid diagrams that represent the same Legendrian link $\Leg \subset (S^3, \xistd)$. Then, there exists a bigraded isomorphism 
  \[ \Phih \colon \GHhb (\grid) \to \GHhb (\grid')\]
  with $\Phih(\lhpb(\grid)) = \lhpb(\grid')$ and $\Phih(\lhmb(\grid)) = \lhmb(\grid')$.
\end{corollary}

Computationally, it is easier to work with the complex $\GCtb(\grid)$ than with $\GHhb (\grid)$. Therefore, we consider the homology classes $\ltpb(\grid),\ltmb(\grid)\in \GHtb(\grid)$ of $\xp(\grid), \xm(\grid)$ and show that they carry the same information as $\lhpb(\grid)$ and $\lhmb(\grid)$:
\begin{proposition}\label{prop:tilde-hat}
There is an injection 
\[\GHhb(\grid)\to \GHtb(\grid)\]
that sends 
$\lhpmb(\grid)$ to $\ltpmb(\grid)$. In particular, $\lhpmb(\Leg) = \lhpmb(\grid)$ is zero if and only if $\ltpmb(\grid)$ is zero.
\end{proposition}

Similarly, one obtains \emph{double-point enhanced transverse link invariants}: 

\begin{theorem}\label{thm:transverse}
  Suppose that $\grid$ and $\grid'$ are two grid diagrams whose corresonding Legendrian links have transversely isotopic positive transverse pushoffs. Then, there exists a bigraded isomorphism 
  \[ \Phi \colon \GHmb (\grid) \to \GHmb (\grid')\]
  with $\Phi(\lpb(\grid)) = \lpb(\grid')$.
\end{theorem}
This justifies writing $\thetab(\Transv)$, where $\Transv$ is the positive transverse pushoff of the Legendrian link $\Leg$ represented by $G$. Similarly, we get an invariant $\thetahb(\Transv)\in \GHhb(\grid)$.

The double-point enhanced Legendrian GRID invariants behave in the same way under stabilization of the Legendrian knot as their counterparts in $\GHm$:
\begin{theorem}
  \label{thm:leg-stab}
  Suppose that $\grid$, $\grid^+$, and $\grid^-$ represent the Legendrian knot $\Leg$ and its positive and negative stabilizations $\Leg^+$ and $\Leg^-$, respectively. Then, there are bigraded isomorphisms
  \[\Phi^-\colon \GHmb(\grid)\longrightarrow \GHmb(\grid^-), \hspace{2cm} \Phi^+\colon \GHmb(\grid)\longrightarrow \GHmb(\grid^+)\]
such that 
\begin{align*}
\Phi^-(\lpb(\grid)) &= \lpb(\grid^-) \quad & U\cdot \Phi^+(\lmb(\grid)) &= \lmb(\grid^+)\\
U\cdot \Phi^-(\lmb(\grid)) &= \lmb(\grid^-)  & \Phi^+(\lpb(\grid)) &= \lpb(\grid^+)
\end{align*}
\end{theorem}
 Descending to the hat version, we obtain the following vanishing result:
\begin{corollary}
\label{cor:stab-vanishing}
If $\Leg$ is the positive stabilization of another Legendrian knot, then $\lhpb(\Leg) = 0$. Similarly, If $\Leg$ is the negative stabilization of another Legendrian knot, then $\lhmb(\Leg) = 0$.
\end{corollary}

\subsection{Obstructions to decomposable Lagrangian cobordisms}
\label{ssec:intro-obstr}

We show that the double-point enhanced GRID invariants obstruct decomposable Lagrangian cobordisms in $(\R_t \cross \R^3, d (e^t \alphastd))$, that is, cobordisms that can be obtained as compositions of elementary exact Lagrangian cobordisms. 

\begin{theorem}
  \label{thm:cob}
  Suppose that $\Leg_-$ and $\Leg_+$ are Legendrian links in $(\R^3, \xistd)$, 
  such that
  \begin{itemize}
    \item $\lhpb(\Leg_+)=0$ and $\lhpb(\Leg_-)\neq 0$; or
    \item $\lhmb(\Leg_+)=0$ and $\lhmb(\Leg_-)\neq 0$.
  \end{itemize}
  Then there does not exist a decomposable Lagrangian cobordism from $\Leg_-$ 
  to $\Leg_+$.
\end{theorem}

Specializing to the case where $\Legm$ is the undestabilizable Legendrian unknot, we obtain the following:

\begin{corollary}
  \label{cor:filling}
  Suppose that $\Leg$ is a Legendrian link in $(\R^3, \xistd)$, such that either
  $\lhpb(\Leg) = 0$ or $\lhmb(\Leg) = 0$.  Then there does not exist a decomposable Lagrangian filling of $\Leg$.
\end{corollary}

As in \cite{BLW22}, we prove \fullref{thm:cob} by showing that the tilde versions of the double-point enhanced GRID invariants satisfy a weak functoriality under decomposable Lagrangian cobordisms:

\begin{theorem}
  \label{thm:functoriality}
  Suppose that $\grid_-$ and $\grid_+$ are two grid diagrams that represent Legendrian links $\Legm$ and $\Legp$ in $(\R^3, \xistd)$ respectively. Suppose that there exists a decomposable Lagrangian cobordism $\Lag$ from $\Legm$ to $\Legp$. Then there exists a  homomorphism
  \[
    \Psi \colon \GHtb (\grid_+) \to \GHtb (\grid_-) \left\llbracket - \chi (\Lag), \frac{\abs{\Leg_+} - \abs{\Leg_-} - \chi (\Lag)}{2} \right\rrbracket,
  \]
  such that 
   \[ \Psi (\ltpmb(\grid_+)) = \ltpmb(\grid_-).
  \]
\end{theorem}

\subsection{Comparison with the GRID invariants}
\label{ssec:intro-compare}

Just like $\lambdahpm$, the double-point enhanced invariants $\lhpmb$ are effective in distinguishing Legendrian knots, distinguishing transverse knots, and obstructing decomposable Lagrangian cobordisms. In fact, we do not know any examples where the two differ. In \cite{T23}, Thakar shows that for a link $L$ with knot Floer thickness at most one, the double-point enhanced grid homology $\GHmb(L)$ is isomorphic to $\GHm(L)[v]$. Thakar's argument does not however imply that the invariants $\lambdahpm$ and $\lhpmb$ agree (i.e,\ that both are zero or both are nonzero). So, it is a priori possible that there is a link for which the two homologies agree, but the respective GRID invariants do not. We have computed $\lambdahpm$ and $\lhpmb$ for almost all Legendrian representatives with  maximal Thurston-Bennequin number for knots of arc index up to $11$.  In all examples we computed, our invariants were nonzero if and only if the (non-enhanced) GRID invariants were nonzero. Further, we show: 
\begin{theorem}
    \label{thm:compare}
    Let $\Leg$ be a Legendrian link, If $\lhpb(\Leg) = 0$, then $\lambdahp(\Leg) =0$. Similarly, if $\lhmb(\Leg) = 0$, then $\lambdahm(\Leg) = 0$.
\end{theorem}

\begin{remark}
The results in this paper, along with an argument analogous to  \cite{bvv}, would imply the existence of a double-point enhanced LOSS invariant in the case of Legendrian links in $(S^3, \xistd)$.
\end{remark}

\subsection{Organization}
\label{ssec:organization}

In \fullref{sec:prelims}, we provide some background on Legendrian and transverse knots,  Lagrangian cobordisms, knot Floer homology, and the GRID invariants. In \fullref{sec:inv}, we define the double-point enhanced GRID classes, and study their behavior under grid commutation and (de)stabilization, ultimately proving that they are Legendrian link invariants and behave as expected under Legendrian stabilization. In \fullref{sec:cob}, we study the behavior of the double-point enhanced GRID invariants under pinches and (the reverse of) births, proving that they satisfy a weak functoriality under decomposable Lagrangian cobordisms and thus obstruct such cobordisms. Finally, in \fullref{sec:comp}, we compare the GRID invariants with the double-point enhanced GRID invariants.

\subsection*{Acknowledgments} We thank Ollie Thakar and Robert Lipshitz for helpful conversations. 
This work is the result of the 2024 Summer Hybrid Undergraduate Research (SHUR) program at Dartmouth College, and the authors thank Dartmouth for the support.  IP was partially supported by NSF CAREER Grant DMS-2145090. The SHUR program was also partially supported by this NSF grants.

\section{Preliminaries}\label{sec:prelims}

In this section, we review some basics about Legendrian and transverse knots,  Lagrangian cobordisms, knot Floer homology, and the GRID invariants.

\subsection{Legendrian Knots and Lagrangian cobordisms}
A link $\Leg \subset \R^3$ is called \emph{Legendrian} if it is everywhere tangent to the standard contact structure on $\R^3$,
\[
  \xi_{\std}=\ker(\alpha_{\std}),\quad \alpha_{\std}=dz-y\,dx.
\]
Two Legendrian links are Legendrian isotopic if they are isotopic through a family of Legendrian links. 

A Legendrian link is often represented by its \emph{front diagram}, or \emph{front projection}, which is its projection onto the $xz$-plane. Note that in a front diagram, strands with lower slope always pass over strands with higher slope. See \fullref{fig:4_1} for an example. 
\begin{figure}[ht]
      \includegraphics[scale=1]{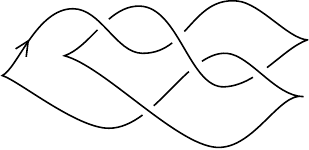}
  \caption{An example of a front projection.}
    \label{fig:4_1}
\end{figure}

Two Legendrian links are Legendrian isotopic if and only if their front diagrams can be related by a sequence of Legendrian planar isotopies (isotopies that preserve left and right cusps) and Legendrian Reidemeister moves (the first three diagrams in \fullref{fig:leg-moves} and their horizontal and vertical reflections). 
\begin{figure}[ht]
      \includegraphics[scale=1]{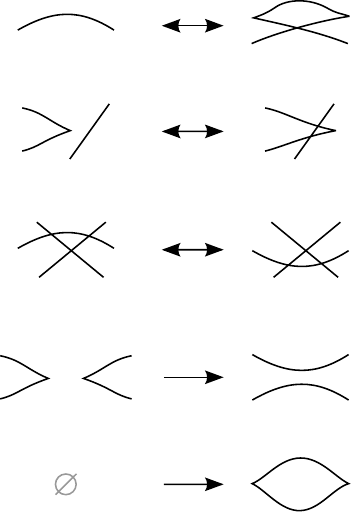}
  \caption{The moves on front projections that correspond to elementary cobordisms; vertical and horizontal reflections of these moves are also allowed. The first three moves are Legendrian Reidemeister moves, the fourth move is called a pinch, and the fifth move is called a birth.}
    \label{fig:leg-moves}
\end{figure}

The two classical Legendrian link invariants are the Thurston--Bennequin number 
$\tb(\Leg)$ and the rotation number $\rot(\Leg)$. These can be computed from an 
oriented front diagram $D$ via the relations
\[
  \tb(\Leg)=\mathrm{wr}(D) - \frac{1}{2} (c_+(D)+c_-(D)), \qquad 
  \rot(\Leg)=\frac{1}{2} (c_-(D)-c_+(D)),
\]
where $\mathrm{wr}(D)$ is the writhe of the diagram (the number of positive crossings minus the number of negative crossings), and $c_-(D)$ and $c_+(D)$ 
are the number of downward and upward cusps, respectively.

A smooth link  $ T \in(\R^3, \xi_{\std})$ is called \emph{transverse} if it is everywhere transverse to the contact structure. Two transverse links are transversely isotopic if they are isotopic through transverse links. A transverse link $T$ naturally inherits an orientation from the (oriented) contact structure: a vector $v$ tangent to $T$ is positive if and only if $\alpha_{\std}(v)>0$.

Given an oriented Legendrian link $\Leg$, we can obtain a transverse link $T_+(\Leg)$, called the \emph{positive transverse pushoff of $\Leg$}, by smoothly pushing $\Leg$ off in a direction transverse to the contact planes so that the orientation of the smooth link is preserved. Further, every transverse link is the positive transverse pushoff of some Legendrian link. 

If the  transverse link  $ T \subset(\R^3, \xi_{\std})$ is  the positive transverse pushoff of a Legendrian link $\Leg$, then \emph{self-linking number $\self(T)$} is defined as
\[\self(T) = \tb(\Leg) - r(\Leg).\]

Transverse links can be studied via front projections similar to those for Legendrian links. As we will instead represent links by grid diagrams in this paper, we do not discuss transverse front projections.

Last, we discuss a certain type of cobordisms between Legendrian links.
The \emph{symplectization} of $(\R^3, \xistd)$ is the symplectic $4$-manifold
\[
  (\R_t \cross \R^3, d (e^t \alphastd)).
\]
Given two Legendrian links $\Legm, \Legp$ in $(\R^3, \xistd)$, 
a \emph{Lagrangian cobordism} from $\Legm$ to $\Legp$ is an oriented, embedded surface $L \subset \R^t \cross 
\R^3$ such that
\begin{itemize}
  \item $L$ is Lagrangian, i.e.\ $d (e^t \alphastd) \rvert_L \equiv 0$;
  \item $L$ has cylindrical ends, i.e.\ for some $T > 0$,
    \begin{align*}
      L \cap ((-\infty, -T) \cross \R^3) &= (-\infty, -T) \cross \Legm,\\
      L \cap ((T, \infty) \cross \R^3) &= (T, \infty) \cross \Legp,
    \end{align*}
    and $L \cap ([-T, T] \cross \R^3)$ is compact.
\end{itemize}
A Lagrangian cobordism is called \emph{exact} if there exists a function $f \colon L \to \R$ that is constant on each of the two cylindrical ends, and satisfies
\[
  (e^t \alphastd) \rvert_L = df.
\]
A \emph{Lagrangian concordance} is a Lagrangian cobordism of genus zero; such a cobordism is automatically exact.

Lagrangian cobordisms, and even concordances \cite{Cha15:NotSym}, are directed. By work of Chantraine \cite{Cha10:Leg}, the existence of a Lagrangian cobordism $L$ from $\Legm$ to $\Legp$ implies that the classical invariants of the two links are related by 
\[tb(\Leg_+)-tb(\Leg_-)=-\chi(L) \qquad \text{and} \qquad r(\Leg_+)-r(\Leg_-)=0,\]
where $\chi(L)$ is the Euler characteristic of $L$. 

An important class of exact Lagrangian cobordisms is the class of \emph{decomposable} Lagrangian cobordisms, that is, cobordisms that are isotopic through exact Lagrangian cobordisms to a composition of elementary exact Lagrangian cobordisms. By work of Bourgeois, Sabloff, Traynor \cite{BST15}, 
Chantraine \cite{Cha10:Leg}, Dimitroglou Rizell \cite{Dim16}, and Ekholm, 
Honda, and K\'alm\'an \cite{EHK16}, there exists an elementary exact Lagrangian cobordism from $\Legm$ to $\Legp$ if and only if  $\Legp$ can be obtained from $\Legm$ by a Legendrian isotopy, a pinch, or a birth; see \fullref{fig:leg-moves}. It is an open question whether every connected, exact Lagrangian cobordism between undestabilizable, nonempty links is decomposable.

\subsection{Knot Floer homology, the GRID invariants, and double points}\label{ssec:background-GH}
 
In this section, we review some of the basics of the combinatorial formulation of knot Floer homology, following the conventions in \cite{OSS15}.

A \emph{grid diagram} $\grid$ is an $n \times n$ grid on the plane, along with two sets of markings $\XX = \{X_1,  \ldots, X_n\}$ and $\OO= \{O_1,  \ldots, O_n\}$ such that each row has one $X$ marking and one $O$ marking, and so does each column, and no square of the grid contains more than one marking. 

A grid diagram $\G$ specifies a link $L(\G)$ in $\R^3$ as follows. Draw horizontal segments from the $O$'s to the $X$'s, vertical segments from the $X$'s to the $O$'s, and require that vertical segments pass over horizontal ones. We say \emph{$\G$ is a grid diagram for $L$}. Conversely, every link $L$ in $\R^3$ can be represented by a grid diagram.
By a theorem of Cromwell \cite{Cro95}, two grid diagrams represent the same link if and only if they are related by a sequence of the following moves: 
\begin{itemize}
\item \emph{cyclic permutations}, in which the topmost row (or bottommost row, rightmost column, leftmost column) is moved to become the bottommost row (or topmost row, leftmost column, rightmost column, resp.);
\item  \emph{commutations}, in which two adjacent rows (or columns) are switched if the corresponding segments connecting the $X$'s and $O$'s are either nested or disjoint;
\item \emph{stabilizations}, in which a $1 \times 1$ square with an $O$ marking (resp.\ $X$ marking)  is replaced by a $2 \times 2$ square with two diagonal $O$ markings and an $X$ marking (resp.\ two diagonal $X$ markings and an $O$ marking), creating a new row and a new column, and \emph{destabilizations}, the inverse operations to stabilizations. 
\end{itemize}
Following \cite{OSS15}, we classify (de)stabilizations by the type of the marker and the location of the empty cell in the $2 \times 2$ square; for example, a stabilization of type \textit{X:SE} results in a $2 \times 2$ square with an empty southeastern cell, an $X$ in the northwestern cell, and $O$'s in the other two cells.

A grid diagram $\G$ also specifies a Legendrian link $\Leg (\G)$ in $(\R^3, \xi_{\std})$, as follows. Start with the
projection of $L(\G)$ onto the grid, smooth 
all northwest and southeast corners of the projection, turn the 
northeast and southwest corners into cusps, rotate the diagram $45$ degrees clockwise, and flip all the crossings.
Note that the smooth type of the Legendrian link $\Leg (\G)$ is the mirror of the smooth link $L(\G)$. Similar to the smooth case, every Legendrian link in $(\R^3, \xistd)$ can be represented by a grid diagram. Two grid diagrams represent the same Legendrian link if and only if they are related by a sequence of cyclic permutations, commutations, and (de)stabilizations of type \textit{X:SE} and \textit{X:NW}; (de)stabilizations of type \textit{O:SE} and \textit{O:NW} result in Legendrian isotopies as well. 
The other types of grid stabilization do change the Legendrian isotopy class of the link: 
stabilizations of type \textit{X:NE} and \textit{O:SW} are positive stabilizations of the Legendrian link, whereas stabilizations of type \textit{X:SW} and \textit{O:NE} are negative stabilizations. 

A grid diagram $\G$ also specifies a transverse link $T(\G)$ in $(\R^3, \xi_{\std})$, by taking the positive pushoff of $\Leg(\G)$. Two grid diagrams represent the same transverse link if and only if they are related by a sequence of cyclic permutations, commutations, and (de)stabilizations of type \textit{X:SE}, \textit{X:NW}, and \textit{X:SW}. (De)stabilizations of type \textit{O:SE},  \textit{O:NW}, and \textit{O:SW} result in transverse isotopies as well. 
Thus, one can think of transverse links up to transverse isotopy as Legendrian links up to Legendrian isotopy and negative stabilization. Grid stabilizations of type \textit{X:NE} and \textit{O:SW} result in stabilization of the transverse link $T(\G)$.

\begin{remark}
Our convention for converting from grid diagrams to Legendrian and transverse knots differs from \cite{CN13atlas} and agrees with \cite{OSS15}, \cite{OST08}, and \cite{NOT08}.
\end{remark}

 To a grid diagram $\grid$ of size $n$ for an $l$-component link $L$, we associate a bigraded chain complex $(\GCm(\grid), \bdm)$ over the ring $\mathbb F \left[V_1, \ldots, V_n\right]$, called the \emph{unblocked grid complex}. The \emph{unblocked grid homology} $\GHm(\grid)$ is the homology of this complex, viewed as a module over $\mathbb F [U_1, \ldots, U_l]$, where the action of $U_i$ is given by multiplication by $V_{k_i}$ for some fixed $k_i$ such that  $V_{k_i}$ corresponds to the $i^{\mathrm th}$ link component. The \emph{simply blocked grid homology} $\GHh(\grid)$ is the homology of the complex $\GCm(\grid)/(V_{k_1}, \ldots, V_{k_l})$, and the \emph{fully blocked gird homology}  $\GHt(\grid)$ is the homology of the complex $\GCm(\grid)/(V_{1}, \ldots, V_{n})$.
 
 Before we define these complexes, we introduce a bit more notation. 

Let $\grid$ be a grid diagram of size $n$. We will think of $\grid$ as a diagram on the torus, by identifying the top and bottom edges, as well as the left and right edges of the grid. The horizontal segments of the grid (which separate the rows of squares) become a set of circles  $\alphas = \{\alpha_1, \ldots ,\alpha_n\}$, indexed from bottom to top,  and the vertical ones become a set of circles $\betas = \{\beta_1, \ldots, \beta_n\}$, 
indexed from left to right.

The module $\GCm(\grid)$ is generated over $\mathbb F \left[V_1, \ldots, V_n\right]$ by \emph{grid states}, that is, bijections between horizontal and vertical circles. Geometrically, a grid state is an $n$-tuple of points $\x = \{x_1, \ldots, x_n\}$ on the torus with one point on each horizontal circle and one on each vertical circle. The set of grid states for a grid diagram $\grid$ is denoted $\SG$.

The bigrading on $\GCm(\grid)$ is induced by two integer-valued functions on the set of grid states, which we define below. Consider the partial ordering of points in $\mathbb{R}^2$ given by $(x_1,y_1)<(x_2,y_2)$ if $x_1<x_2$ and $y_1<y_2$. For any two sets $P,Q \subset \mathbb{R}^2$, define \[\mathcal{I}(P,Q) = \#\setc{(p,q) \in P\times Q}{p<q},\] and symmetrize this function, defining
\[\J(P,Q)=\frac{\I(P,Q)+\I(Q,P)}{2}.\] 
Consider a fundamental domain  $[0,n) \times [0,n) \subset \mathbb{R}^2$ for the torus. We can think of 
a grid state $\x \in S(\G)$ as a set of points with integer coordinates, and $\X$ and 
$\OO$ as sets of points with half-integer coordinates in the fundamental domain. Given a grid state $\x$, define
\begin{align*}
   M_{\OO}(\x) & =\J(\x,\x)-2\J(\x,\OO)+\J(\OO,\OO)+1, \\
    M_{\X}(\x) & =\J(\x,\x)-2\J(\x,\X)+\J(\X,\X)+1.
\end{align*}
Finally, the Maslov and Alexander functions $M (\x)$ and $A(\x)$ are defined as
\begin{align*}
    M (\x) &= M_{\OO}(\x) \\
    A(\x)&=\frac{1}{2}\Big(M_\OO(\x)-M_\X(\x)\Big)-\frac{n-l}{2}.
\end{align*}
Extend the Maslov and Alexander functions to a bigrading on $\GCm(\G)$ by 
\begin{align}
  M (V_1^{k_1}\cdots V_n^{k_n}\x) &= M(\x) - 2k_1-\cdots - 2k_n\\
  A (V_1^{k_1}\cdots V_n^{k_n}\x) &=  A(\x) - k_1-\cdots - k_n. 
\end{align}

Given two grid states $\x,\y \in \SG$, let $\Rect(\x, \y)$ denote the set of rectangles embedded in the torus with the following properties. First, $\Rect (\x, \y)$ is empty if $\x$ and $\y$ do not agree at exactly $n - 2$ points. An element $r \in \Rect (\x, \y)$ is an embedded rectangle with right angles, such that:
\begin{itemize}
\item $\bdy r$ lies on the union of horizontal and vertical circles;
\item The vertices of $r$ are exactly the points in $\x \triangle \y$, where $\triangle$ denotes the symmetric difference; and
\item $\bdy (\bdy r \intersect \betas)= \x-\y$, in the orientation induced by $r$.
\end{itemize}
Given $r\in\Rect(\x, \y)$, we say that \emph{$r$ goes from $\x$ to $\y$.} Observe that $\Rect(\x, \y)$  consists of either zero or two rectangles. We say a rectangle $r\in \Rect(\x, \y)$ is \emph{empty} if $\x\cap \Int(r) = \y\cap \Int(r) = \emptyset$. We denote the set of empty rectangles from $\x$ to $\y$ by $\eRect(\x, \y)$.

Define the \emph{multiplicity} $O_i(r)$ of $r$ at the marking $O_i$ as one if $r$ contains $O_i$ and zero otherwise. 
The differential on $\GCm(\G)$ is defined on generators by 
\[\bdm(\x) = \sum_{\y\in \SG}\sum_{\substack{r\in \eRect(\x, \y)\\ r\cap \XX = \emptyset}} V_{1}^{O_{1}(r)}V_{2}^{O_{2}(r)}\cdots V_{n}^{O_{n}(r)}\cdot \y.\]

One could also consider the simpler, \emph{fully blocked grid complex} $\GCt(\grid)$ generated over $\mathbb F$ by $\SG$ with differential 
\[\bdtox(\x) = \sum_{\y\in \SG}\sum_{\substack{r\in \eRect(\x, \y)\\ r\cap \XX = \emptyset = r\cap \OO}} \y.\]
The \emph{fully blocked grid homology} $\GHt(\grid)$ is the homology of $(\GCt(\grid), \bdtox)$.

Note that for any two states $\x,\y$ with a rectangle $r\in \Rect(\x,\y)$, we have
\begin{align}
  M (\x) - M (\y) &= 1 - 2 \#(r \cap \OO)+ 2 \# (\Int(r) \cap \x), \label{eq:maslov-gr-rel}\\
  A (\x) - A (\y) &= \# (r \cap \XX) - \# (r \cap \OO), \label{eq:alex-gr-rel}
\end{align}
so $\bdm(\x)$ and $\bdtox$ are homogeneous of $(M,A)$ bidegree $(-1,0)$.

Last,  the \emph{filtered grid complex} $\fGCt(\grid)$ is the complex with underlying module $\GCt(\grid)$, differential 
\[\bdt(\x) = \sum_{\y\in \SG}\sum_{\substack{r\in \eRect(\x, \y)\\ r\cap \OO = \emptyset}} \y,\]
and filtration induced by the Alexander grading. Note that the homology of the associated graded object is $H_*(\gr(\fGCt(\G))) = \grGH(\G)$.

 A different, less studied version of knot Floer homology was introduced by Lipshitz \cite{L06}, and is known as \emph{double-point enhanced knot Floer homology}. In its combinatorial version, one counts rectangles that are not necessarily empty, and uses a new formal variable to track non-emptiness. We recall the combinatorial formulation below; see also \cite[Section~5.5]{OSS15}.

Once again, let $\grid$ be a grid diagram of size $n$ for an $l$-component link $L$.
 
 The simplest double-point enhanced grid complex is the fully blocked one, in which one counts (not necessarily empty) rectangles with no $X$ or $O$ markings in their interior: 
\begin{definition} 
The \emph{fully blocked double-point enhanced grid complex} $\GCtb(\grid)$ is the free module over $\mathbb{F}[v]$ generated by $\SG$ and equipped with the differential $\bdtoxb$, where
\[ \bdtoxb(\x)=\sum_{\y\in\SG}\ \sum_{\substack{r\in\Rect(\x,\y)\\r\cap \XX=\emptyset=r\cap\OO}} v^{\#(\Int(r)\cap\x)}\y. \]
\end{definition}
More generally, one has the unblocked version, where only $X$ markings are forbidden: 
\begin{definition} 
The \emph{unblocked double-point enhanced grid complex} $\GCmb(\grid)$ is the free module over $\mathbb{F}[V_1,V_2,\dots,V_n,v]$ generated by $\SG$ and equipped with the differential
\[ \bdmb(\x)=\sum_{\y\in\SG}\ \sum_{\substack{r\in\Rect(\x,\y)\\r\cap \XX=\emptyset}} V_{1}^{O_{1}(r)}V_{2}^{O_{2}(r)}\cdots V_{n}^{O_{n}(r)}v^{\#(\Int(r)\cap\x)}\y. \]
\end{definition}
The \emph{unblocked double-point enhanced grid homology $\GHmb(\grid)$}  is the homology of the complex $(\GCmb(\grid), \bdmb)$, viewed as a module over $\mathbb F [U_1, \ldots, U_l]$, where the action of $U_i$ is given by multiplication by $V_{k_i}$ for some fixed $k_i$ such that  $V_{k_i}$ corresponds to the $i^{\mathrm th}$ link component. 
 The \emph{simply blocked  double-point enhanced grid homology} $\GHhb(\grid)$ is the homology of the complex $\GCmb(\grid)/(V_{k_1}, \ldots, V_{k_l})$. Last, the \emph{fully blocked double-point enhanced grid homology}  $\GHtb(\grid)$ is the homology of the complex $\GCmb(\grid)/(V_{1}, \ldots, V_{n})\cong \GCtb(\grid)$.

The Maslov and Alexander functions induce a bigrading on the double-point enhanced grid complexes, so that multiplication by $v$ increases the Maslov grading by two and preserves the Alexander grading, i.e. 
\begin{align}
  M (V_1^{k_1}\cdots V_n^{k_n}v^k\x) &= M(\x) - 2k_1-\cdots - 2k_n+ 2k\\
  A (V_1^{k_1}\cdots V_n^{k_n}v^k\x) &=  A(\x) - k_1-\cdots - k_n. 
\end{align}

The homologies $\GHm(\grid)$ and $\GHh(\grid)$, as well as their double-point enhanced counterparts $\GHmb(\grid)$ and $\GHhb(\grid)$, are invariants of the underlying link $L$. The homologies $\GHt(\grid)$ and $\GHtb(\grid)$ are almost invariants of the link $L$; specifically, there are isomorphisms

\begin{equation*}
    \GHt(\G) \cong \GHh(\grid)\otimes W^{\otimes (n-l)} \qquad \text{and} \qquad \GHtb(\G) \cong \GHhb(\grid)\otimes W^{\otimes (n-l)} ,
\end{equation*}
where $W$ is the two-dimensional bigraded vector space with one generator in bigrading $(0,0)$ and another in bigrading $(-1,-1)$. 

We also introduce  the \emph{filtered double-point enhanced grid complex} $\fGCtb(\grid)$ as the complex with underlying module $\GCtb(\grid)$, differential 
\[\bdtb(\x) = \sum_{\y\in \SG}\sum_{\substack{r\in \Rect(\x, \y)\\ r\cap \OO = \emptyset}} \y,\]
and filtration induced by the Alexander grading. While it is easy to see that the homology of the associated graded object is $H_*(\gr(\fGCtb(\G))) = \GHtb(\G)$, and is hence a link invariant, we do not at present know whether the filtered chain homotopy type of $\fGCtb$ is a link invariant.

Given a grid diagram $\grid$, the \emph{canonical grid states} $\xp(\grid)$ and $\xm(\grid)$ are composed of the points directly to the northeast of the $X$ markings  and the points directly to the southwest of the $X$ markings, respectively. Their homology classes in $\GHm(\grid)$ are invariants of the Legendrian link $\Leg$ represented by $\grid$, denoted $\lp(\Leg)$ and $\lm(\Leg)$, respectively, and obstruct decomposable Lagrangian cobordisms. The homology class of $\xp(\grid)$ is also an invariant of the transverse link $T$ represented by $\grid$ and is denoted $\theta(T)$. In the next two sections, we prove that the canonical grid states yield invariants and obstructions when considered as elements in double-point enhanced grid homology as well.

\section{The double-point enhanced GRID invariants}\label{sec:inv}

\subsection{Definition of the double-point enhanced GRID invariants}
We now define our invariants in the double-point enhanced hat, minus, and tilde theories: 
\begin{definition}
Suppose $\grid$ is a grid diagram. 
We define $\lhpb(\grid)$, $\ltpb(\grid)$, and $\lpb(\grid)$ to be the homology classes of $\xp(\grid)$ in $\GHhb(\grid)$, $\GHtb(\grid)$, and $\GHmb(\grid)$, respectively. Define $\lhmb(\grid)$, $\ltmb(\grid)$, and $\lmb(\grid)$ analogously, replacing $\xp(\G)$ with $\xm(\G)$.
\end{definition}

Now, consider the filtered double-point enhanced complex $\fGCtb(\grid)$, and let $(E^r_p)$ be the associated spectral sequence, where $r$ is the index of the page and $p$ is the grading induced by the Alexander filtration, as in \cite[Section 2.3]{JPSWW22}. We define filtered double-point enhanced GRID (conjectured) invariants, analogous to those in \cite{JPSWW22}:

\begin{definition}
Suppose $\grid$ is a grid diagram, and let $A = A(\xp(\grid))$. 

We define $\npb(\G)$ to be the smallest integer $i$ for which $d^i_A [\xp(\G)]^i \neq 0 \in E^i_A$, or $\infty$ if $d^i_A [\xp(\G)]^i = 0$ for all $i \in \Z_{\geq 1}$. We define $\nmb(\G)$ analogously, replacing $\xp(\G)$ with $\xm(\G)$.

For each $1 \leq i \leq \np(\G)$, we define
\[
  {\lpb}_i(\G) = [\xp(\G)]^i \in E^i_{A(\xp(\G))}, \qquad {\lmb}_i(\G) = [\xm(\G)]^i \in E^i_{A(\xm(\G))}.
\]
\end{definition}

Our proofs of invariance in this section, as well as the proofs of obstructions to exact Lagrangian cobordisms in the next section, heavily rely on the use of shapes with various geometric properties on a grid in order to define relevant maps, just like we used rectangles to define the differentials in \fullref{ssec:background-GH}. We will introduce each of these shapes at the time that they become necessary. They are all instances of what is called a ``domain'' from one grid state to another. 
 Given $\x, \y \in \SG$, a \emph{domain $\psi$ from $\x$ 
      to $\y$} is a formal linear combination of the closures of the squares in 
    $\grid$, such that $\bdy (\bdy \psi \cap \betas) = \x - \y$ in the induced 
    orientation on $\bdy \psi$. The set of all domains from $\x$ to $\y$ is 
    denoted by $\pi (\x, \y)$. See, for example, \cite[Definition~4.6.4]{OSS15}.

\subsection{Invariance under commutations}
\label{sec:comm}
In this section, we prove invariance under commutation in the unblocked  $\GHmb$ theory as well as in the filtered $\fGCtb$ theory.

\begin{lemma}\label{lem:comm}
Suppose $\grid'$ is obtained from $\grid$ by a commutation move. Then there exists a filtered chain homomorphism 
\[ \commmapb:\fGCtb(\grid)\to \fGCtb(\grid') \]
such that $\commmapb(\x^{\pm}(\grid))=\x^{\pm}(\grid')+\y'$, where $\y'\in\fGCtb(\grid')$ is in a lower filtration level than $\x^\pm(\grid')$. 

Similarly, there exists a bigraded quasi-isomorphism
\[ \commmapxb:\GCmb(\grid)\to\GCmb(\grid') \]
such that $\commmapxb(\x^{\pm}(\grid))=\x^{\pm}(\grid')$.
\end{lemma}

The proof of \fullref{lem:comm} is similar to that of \cite[Lemma 3.3]{JPSWW22}. Before we can begin, we need to introduce a ``combined'' diagram for $\grid$ and $\grid'$ and define certain domains on it. The cases of column and row commutations are analogous, so we will suppose that the commutation is a row commutation, and consider the combined diagram for $\grid$ and $\grid'$ on \fullref{fig:comm}.
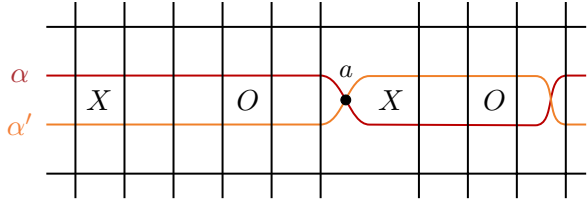
\begin{figure}[ht]
         \labellist
    \pinlabel $O$ at 76 37
    \pinlabel $X$ at 20 37
    \pinlabel $O$ at 169 37
    \pinlabel $X$ at 130 37
   \pinlabel {\small{$a$}} at 113 48
    \pinlabel \textcolor{Maroon}{$\alpha$} at -10 46
    \pinlabel \textcolor{Orange}{$\alpha'$} at -10 28
    \endlabellist
  \includegraphics[scale=1]{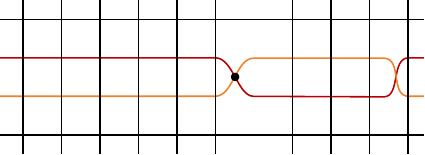}
  \caption{The combined diagram of a row commutation involving $\alpha$ and $\alpha'$.}
    \label{fig:comm}
\end{figure}
We will define the maps $\commmapb$ and $\commmapxb$ by counting certain pentagons on this combined diagram. 

\begin{definition}[cf.\ \cite{T23} Definition 3.6]\label{def:pent}
Let $\x \in S(\grid)$ and $\y' \in S(\grid')$.

A \emph{pentagon $p$ from $\x$ to $\y'$} is an embedded pentagon with non-reflex angles in the combined diagram whose boundary lies on the curves coming from the horizontal and vertical circles of $\grid$ and $\grid'$ and whose vertices are the points $(\x \triangle \y) \cup \{a\}$, where $\triangle$ denotes the symmetric difference, such that $\bdy (\bdy p \cap \betas) = \x - \y'$ in the induced orientation. Let $\Pent(\x,\y')$ be the space of pentagons from $\x$ to $\y'$ and note that $\Pent(\x,\y)$ is empty if $\x$ and $\y$ do not agree at exactly $n-2$ points. 

A \emph{long pentagon $p$ from $\x$ to $\y'$} is an embedded pentagon in the universal cover of the combined toroidal diagram satisfying the above properties (modified appropriately, namely,  we require that the vertices project to $(\x \triangle \y) \cup \{a\}$ and that the signed set $\bdy (\bdy p \cap \betas)$ projects to $\x - \y'$), along with the requirement that the pentagon has height one and that its projection, which we also denote $p$, onto the torus has at least one point of multiplicity $2$ and at least one point of multiplicity $1$. Let $\Pentl$ denote the space of long pentagons from $\x$ to $\y'$.

Let $\Pentx$ and $\Pentxl$ be the spaces of pentagons and long pentagons whose interiors contain no $X$ markings. Let $\Pento$ and $\Pentol$ be defined similarly for $O$ markings.
\end{definition}

\begin{proof}[Proof of \fullref{lem:comm}]
As mentioned earlier, we suppose the commutation is row commutation, and consider the combined diagram on \fullref{fig:comm}. We define the maps $\commmapb$ and $\commmapxb$ by counting pentagons and long pentagons that contain no $O$ markings and no $X$ markings, respectively, in their interiors: 
\begin{align*} 
\commmapb(\x) &=\sum_{\y'\in\boldsymbol{S}(\mathbb{G'})}\sum_{p\in\Pento}  v^{\#(\Int(p)\cap\x)} \cdot \y' + \sum_{\y'\in\boldsymbol{S}(\mathbb{G'})}\sum_{p\in\Pentol}  v \cdot \y'\\
 \commmapxb(\x)&=\sum_{\y'\in\boldsymbol{S}(\grid')}\sum_{p\in\Pentx}  V_{1}^{O_{1}(p)}V_{2}^{O_{2}(p)}\cdots V_{n}^{O_{n}(p)}v^{\#(\Int(p)\cap\x)}\cdot \y'\\
 &\quad +\sum_{\y'\in\boldsymbol{S}(\grid')}\sum_{p\in\Pentxl}  V_{1}^{O_{1}(p)}V_{2}^{O_{2}(p)}\cdots V_{n}^{O_{n}(p)}v\cdot \y'.
\end{align*}

By doing a case analysis similar to that in \cite[Lemma 4.1]{JPSWW22}, we can see that $\commmapxb$ is a bigraded map (also stated in \cite[Proposition 5.4]{T23}), and that $\commmapb$ preserves the Maslov grading and respects the Alexander filtration.  For example, one can check that, given $\x \in \SG$ and $\y' \in S(\grid')$,  if $p\in \Pento\cup \Pentol$, then $M(\y') = M(\x)-2$, and $A(\x) - A(\y') = \#(p\cap\XX)$.

To show that $\commmapb$ and $\commmapxb$ are chain maps, we show that juxtapositions of a pentagon and a rectangle cancel out in pairs. The case of $\commmapxb$  is covered in \cite{T23}, as a result of a number of smaller lemmas and propositions. We include a complete argument here that covers both $\commmapb$ and $\commmapxb$, in order to keep the proof self-contained for the reader and to resolve some small inaccuracies from \cite{T23}. 

Let $\x \in \SG$ and $\z' \in S(\grid')$. Let $\psi\in \pi(\x,\z')$ be a domain that can be decomposed as the juxtaposition of a rectangle and a (possibly long) pentagon, or a (possibly long) pentagon and a rectangle. 
For the case of  $\commmapb$  assume that the interior of $\psi$ contains no $O$ markings, and for the case of $\commmapxb$ assume it contains no $X$ markings. It is clear that the initial and final grid states must differ at $1$, $3$, or $4$ points on the combined diagram. We consider these three cases below. 

\begin{enumerate}
\item Suppose $|\x\setminus \z'| = 4$.  

In this case, the domain can be decomposed in exactly two ways: as $p_1 \ast r_1$ and $r_2 \ast p_2$ where $r_1$ and $r_2$ are rectangles (in $\grid$ and $\grid'$, respectively) with the same support, and $p_1$ and $p_2$ are pentagons with the same support. For the case of $\commmapxb$, note that $O_i(r_1)= O_i(r_2)$ and $O_i(p_1)= O_i(p_2)$ for each $i$, so the two juxtapositions contribute the same power of $V_i$ to the coefficient of $\z'$ in $\commmapxb \circ \bdmb(\x) + \bdmb \circ \commmapxb(\x)$. 

First, suppose the pentagons are long. Then they each contribute a power of $v$. Further, since the pentagons are thin, each of their vertical edges is either fully in the interior of the support of the rectangles, or does not overlap the rectangles at all. So the order of the juxtaposition does not affect the number of points of the starting grid state in the interior of each rectangle. Thus, the two rectangles contribute the same power of $v$ to the differential as well.
 
 Now, suppose the pentagons are not long. A case analysis of the number of corners of each shape in the interior of the other shows that the two juxtapositions contribute the same total power of $v$ to the coefficient of $\z'$; the analysis is analogous to that in \cite[Lemma 4.2]{T23}, where the two shapes are rectangles.
 
Thus, the coefficient of $\z'$ in $\commmapb \circ \bdtb(\x) + \bdtb \circ \commmapb(\x)$, and similarly in $\commmapxb \circ \bdmb(\x) + \bdmb \circ \commmapxb(\x)$, is zero.

\item Suppose $|\x\setminus\z'|=3$.

In this case, the two shapes share exactly one corner. The domain has exactly one other decomposition as a juxtaposition of a rectangle and a pentagon (not necessarily in the same order).  The powers of $V_i$ in the coefficient of $\z'$ in $\commmapxb \circ \bdmb(\x) + \bdmb \circ \commmapxb(\x)$ coming from the two decompositions are the same, since the two decompositions have the same domain, and multiplicities of $O_i$ are additive under juxtaposition. It remains to analyze the contributions of the two decompositions to the power of $v$ in the coefficient of $\z'$.

If $\psi$ can be embedded into a fundamental domain for the torus, then neither rectangle contains a pentagon corner in its interior, and vice versa. Thus, the two decompositions contribute the same to the power of $v$.

Otherwise, at least one decomposition has a long pentagon. \footnote{We caution the reader about a small inaccuracy in the statement of \cite[Corollary 4.1]{T23}, where it is stated that the presence of a long pentagon implies the other pentagon is also long.} If both pentagons are long, then neither rectangle contains a pentagon corner in its interior, each pentagon contributes one to the power of $v$, and the two rectangles contribute equally. 
If both decompositions have long pentagons, neither rectangle contains a pentagon corner in its interior. Last, supposed only one decomposition has a long pentagon. Let $r$ and $p$ be the rectangle and long pentagon in the one decomposition, and $r'$ and $p'$ be the rectangle and pentagon in the other decomposition. Then $p$ contributes one to the power of $v$ and $p'$ does not contribute. Meanwhile, $r'$ contains a corner of $p'$ in its interior, and contributes one more than $r$ to the power of $v$. Thus, the two decompositions contribute the same to the power of $v$. See \fullref{fig:longpent}.

\item Suppose $|\x\setminus\z'|=1$. 

Observe that the complement of the $\alphas$-circles on the combined diagram consists of $n-2$ annuli containing one $X$ and one $O$ each, $2$ annuli with no markings in their interior (each having an arc of $\alpha$ and an arc of $\alpha'$ on their boundary), and two bigons (each having arc of $\alpha$ and an arc of $\alpha'$ as its boundary). 

The domains in this third case are precisely those consisting of a thin annulus with no markings and a triangle (a piece of a bigon cut off by a $\betas$-circle). In this case, there is exactly one way to decompose $\psi$ as a juxtaposition of a pentagon and rectangle. However, these is a unique other domain $\psi\in \pi(\x,\z')$ of this type that decomposes as a juxtaposition of a pentagon and rectangle; namely, the domain whose support consists of the same triangle as $\psi$ together with the opposite annulus with no markings. Thus, domains of this third type cancel each other in pairs. In fact, these are exactly the domains seen in \cite[Lemma 5.4.1 (P-3)(v)]{OSS15}, with appropriate allowed of markings in the bigons depending on the chain complex.

Note that these domains do not contribute a power of $v$, and their contribution to the power of $V_i$ is defined by the possible $O$ marking in the triangle, hence canceling pairs contribute the same to the power of $V_i$.

\end{enumerate}
\begin{figure}
    \includegraphics{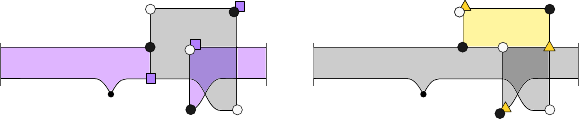}
    \caption{An example of a domain with two decompositions into a rectangle and a pentagon, where exactly one of the pentagons is long.}
    \label{fig:longpent}
\end{figure}

That $\commmapxb$ is a quasi-isomorphism is shown in  \cite[Proposition 5.15]{T23}; the argument relies on proving a homotopy formula by defining a map that counts hexagons.

Finally, $\commmapb(\x^{\pm}(\grid))=\x^{\pm}(\grid')+\y'$ and $\commmapxb(\x^{\pm}(\grid))=\x^{\pm}(\grid')$ since there exists exactly one $X$-free pentagon from $\x^{\pm}(\grid)$, and it is entirely contained within a thin annulus thus is empty and has no $O$ markings.
\end{proof}

\subsection{Invariance under stabilization and destabilization}

In this section, we prove invariance under stabilization and destabilization (see \fullref{fig:stab}) in the unblocked  $\GHmb$ theory. Unfortunately, we do not have a proof in the filtered $\fGCtb$ theory at present; recall, in fact, that it is not even known if the filtered chain homotopy type of $\fGCtb$ is a link invariant.

\begin{figure}[ht]
        \labellist
            \pinlabel $X$ at 12 23
            \pinlabel $O_1$ at 102 33
            \pinlabel $X_2$ at 102 12
            \pinlabel $X_1$ at 123 33
            \pinlabel {\small{$c$}} at 117 17
            \pinlabel $O$ at 188 12
            \pinlabel $X$ at 167 12
            \pinlabel $X$ at 188 33
            \pinlabel {\small{$c$}} at 182 17
            \pinlabel $O$ at 254 33
            \pinlabel $X$ at 254 12
            \pinlabel $X$ at 232 33
            \pinlabel {\small{$c$}} at 247 17
            \pinlabel $O$ at 297 12
            \pinlabel $X$ at 319 12
            \pinlabel $X$ at 297 33
            \pinlabel {\small{$c$}} at 312 17
        \endlabellist
  \includegraphics[scale=1]{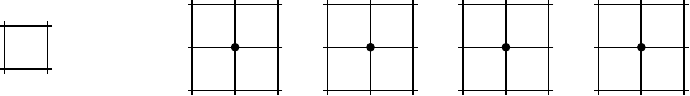}
  \caption{Left: an $X$-marked square in $\grid$. The distingushed 2x2 square of the diagram  $\grid'$ obtained from $\grid$ by stabilization at this $X$ marking, where the stabilization is of type \textit{X:SE}, \textit{X:NW}, \textit{X:SW}, and \textit{X:NE}, as seen from left to right.}
  \label{fig:stab}
\end{figure}
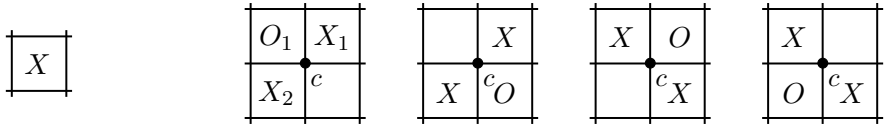

\begin{lemma}\label{lem:stab} 
Suppose $\grid'$ is obtained from $\mathbb G$ by a stabilization. 
    \begin{itemize}
    \item[\mylabel{itm:s1}{(S-1)}] 
        If $\grid'$ is obtained from $\grid$ by a stabilization of type \textit{X:SE} or \textit{X:NW}, then there exists an isomorphism $\GHmb(\grid') \isom \GHmb (\grid)$ of bigraded $\mathbb{F}[U]$-modules that takes $\lpb (\grid')$ to $\lpb (\grid)$ and takes $\lmb (\grid')$ to $\lmb(\grid)$.
    \item[\mylabel{itm:s2}{(S-2)}]
        If $\grid'$ is obtained from $\grid$ by a stabilization of type \textit{X:SW}, then there exists an isomorphism $\GHmb(\grid') \isom \GHmb (\grid)$ of bigraded $\mathbb{F}[U]$-modules that takes $\lpb(\grid')$ to $\lpb(\grid)$ and takes $\lmb(\grid')$ to $U \cdot \lmb(\grid)$.
    \item[\mylabel{itm:s3}{(S-3)}]
        If $\grid'$ is obtained from $\grid$ by a stabilization of type \textit{X:NE}, then there exists an isomorphism $\GHmb(\grid') \isom \GHmb (\grid)$ of bigraded $\mathbb{F}[U]$-modules that takes $\lpb(\grid')$ to $U \cdot \lpb(\grid)$ and takes $\lmb(\grid')$ to $\lmb(\grid)$.
    \end{itemize}
\end{lemma}

\begin{proof}
We begin with Part~\ref{itm:s1}, focusing on the case \textit{X:SE}. Let $O_1$ be the new $O$ marking, let $X_2$ be the new $X$ marking directly below $O_1$, and let $O_2$ be the $O$ marking in the same row as $X_2$. 
With this labeling, $\GCmb(\grid')$ is a chain complex over $\mathbb F \left[V_1, \ldots, V_n,v\right]$, while $\GCmb(\grid)$ is a chain complex over $\mathbb F \left[V_2, \ldots, V_n,v\right]$, where $n$ is the size of the grid $\grid'$.
We generalize the proof of \cite[Lemma 6.4.6 Case (S-1)]{OSS15} to double-point enhanced grid homology. Similarly to the previous section, we again need to introduce maps that also count long domains.  

Let $c$ be the intersection point of the two new curves in $\grid'$. Partition the grid states $S(\grid')$ into $I(\grid')\sqcup N(\grid')$, where $I(\grid')$ consists of the grid states that contain $c$, and $N(\grid')$ consists of the grid states that do not contain $c$. Write $I$ and $N$ for the corresponding $\mathbb F \left[V_1, \ldots, V_n, v\right]$-submodules of $\GCmb(\grid')$ and write the differential $\bdmb$ on $\GCmb(\grid') \cong N\oplus I$ as a $2\times 2$ matrix
  \[
  \bdmb = 
  \begin{pmatrix}
    \bdy_I^I & \bdy_I^N \\
    0 & \bdy_N^N
  \end{pmatrix}.
  \]
In this notation, we see that $\GCmb(\grid')$ is the mapping cone of $\bdy_I^N\colon (N, \bdy_N^N) \to (I, \bdy_I^I)$.

As for $\GCmb(\grid)$, \cite[Lemma 5.2.16]{OSS15} implies that on the level of homology we have a bigraded isomorphism of $\mathbb F \left[V_2, \ldots, V_n, v\right]$-modules
\[H(\Cone(V_1-V_2\colon \GCmb(\grid)[V_1]\to \GCmb(\grid)[V_1]))\cong \GHmb(\grid).\]
Furthermore, for a cycle $x\in \GCmb(\grid)$ this isomorphism identifies the homology class $[x]\in \GHmb(\grid)$ with the homology class $[(0,x)]\in H(\Cone(V_1-V_2\colon \GCmb(\grid)[V_1]\to \GCmb(\grid)[V_1]))$, where $(0,x)\in \GCmb(\grid)[V_1]\oplus \GCmb(\grid)[V_1]$. Thus, to complete the proof, we just need to exhibit a quasi-isomorphism from $\Cone (V_1-V_2)$ to $\Cone(\bdy_N^I)$ that takes $(0,\xp)$ to $(0,\xp)$ and $(0,\xm)$ to $(0,\xm)$. To do this, we will define chain homotopy equivalences
\[
e' \colon \GCmb(\grid)\to I, \qquad
\hxnb \colon I\to N
\]
that give rise to a commutative diagram
  \begin{equation}
    \xymatrix{
      N  \ar[rr]^{\bdy_N^I}
      & & I \\
      & & \\
      \GCmb(\grid)[V_1]\left\llbracket 1, 1 \right\rrbracket \ar[rr]^{V_1-V_2} \ar[uu]^{\hxnb\circ e'}
      & & \GCmb(\grid)[V_1] \ar[uu]_{e'}
    }
     \label{diag:X:SE}
  \end{equation}
with $e'(\xp)=\xp$ and $e'(\xm) = \xm$.    
By \cite[Lemma 5.2.12]{OSS15}, these maps will induce the desired quasi-isomorphism. 

 We proceed with defining the maps $e'$ and $\hxnb$ and establishing their properties. The map $e'$ is the isomorphism induced by the natural identification of $\SG$ and $I$. The map $\hxnb\colon I\to N$ is the $\mathbb F [V_1, \ldots, V_n, v]$-module homomorphism defined on grid states by counting (not necessarily empty) rectangles whose interior intersects $\XX$ in $X_2$: 
\[
 \hxnb(\x) = \sum_{\y\in N(\grid')} \sum_{\substack{r\in \Rect(\x,\y)\\ r\cap \XX=X_2}} V_{1}^{O_{1}(r)}V_{2}^{O_{2}(r)}\cdots V_{n}^{O_{n}(r)}v^{\#(\Int(r)\cap\x)}\cdot \y. 
 \]

Since $e'$ is an isomorphism, we only need to show that $\hxnb \colon I\to N$ is a chain homotopy equivalence.

We first show that $\hxnb$ is a chain map, i.e.,
\[\hxnb\circ \bdy_I^I+\bdy_N^N\circ \hxnb = 0.\]

To do this, we will show that juxtapositions of rectangles that contribute to $\hxnb\circ \bdy_I^I+\bdy_N^N\circ \hxnb \colon I\to N$ cancel in pairs.

Let $\x\in I$ and $\y\in N$. Let $\psi\in\pi(\x,\y)$ be a domain that can be decomposed as $r_{1}\ast r_{2}$ such that one of $r_{1},r_{2}$ contributes to $\hxnb$ and the other to $\bdy_I^I$ or $\bdy_N^N$. 
Note that $\x$ and $\y$ must differ at least at one point, as $c$ is a coordinate of $\x$ and not of $\y$. Since  $r_{1}$ and $r_{2}$ are rectangles, $\x$ and $\y$ must then differ at $4$ or $3$ points. We consider the two cases below.

\begin{enumerate}
\item Suppose $|\x\setminus\y| = 4$.

In this case, the domain has a unique alternate decomposition $r_{3}\ast r_{4}$, where $r_2$ and $r_3$ have the same support,  $r_1$ and $r_4$ have the same support, and $r_{1}\ast r_{2}$ and $r_{3}\ast r_{4}$ contribute to opposite summands of $\hxnb\circ \bdy_I^I (\x)+\bdy_N^N\circ \hxnb(\x)$. 
As in the proof of \fullref{lem:comm}, we see that the two juxtapositions contribute the same total powers of $v$ and of each $V_i$ to the coefficient of $\y$; for $v$, the case analysis of the number of corners of each shape in the interior of the other is covered by \cite[Lemma 4.2]{T23}.
Thus, the coefficient of $\y$ in $\hxnb\circ \bdy_I^I (\x)+\bdy_N^N\circ \hxnb(\x)$ is zero.

\item Suppose $|\x\setminus\y| = 3$.

In this case, $\phi$ can be embedded into a fundamental domain for the torus and has a one reflex angle. The domain has a unique alternate decomposition $r_{3}\ast r_{4}$, by cutting the other way at the reflex angle, which also contributes to $\hxnb\circ \bdy_I^I (\x)+\bdy_N^N\circ \hxnb(\x)$.
The two decompositions contribute equal coefficients to $\y$ by \cite[Lemma 4.3]{T23}\footnote{We remark on a small inaccuracy in \cite[Lemma 4.3]{T23}: the claim ``$\mathcal T = 0$'' is incorrect, but does not affect the validity of the general argument.}.
\end{enumerate}

Next, we will show that $\hxnb$ is a chain homotopy equivalence.

Define $\mathbb F \left[V_1, \ldots, V_n,v\right]$-module maps $\homb \colon N \to I$ and $\hoxmb \colon N \to N$ as follows. The map  $\homb$ is defined by
\[
\homb(\x) = \sum_{\y\in I(\grid')} \sum_{\substack{r\in\Rect(\x,\y)\\O_1 \in r, \, \Int(r)\cap \mathbb{X}=\emptyset}} V_{2}^{O_{2}(r)}\cdots V_{n}^{O_{n}(r)}v^{\#(\Int(r)\cap\x)}\cdot \y.
\]
To define $\hoxmb$, we need to introduce one new type of rectangle. Given $\x, \y\in \grid'$, a \emph{long rectangle $r$ from $\x$ to $\y$} is an embedded rectangle in the universal cover of the toroidal diagram satisfying the same properties as rectangles in $\Rect(\x, \y)$ (modified appropriately, namely,  we require that the vertices project to $\x \triangle \y$ and that the signed set $\bdy (\bdy r \cap \betas)$ projects to $\x - \y$), along with the requirement that the rectangle has width one and that its projection, which we also denote $r$, onto the torus has at least one point of multiplicity $2$ and at least one point of multiplicity $1$; cf.\ \cite[Definition 3.2]{T23}. Let $\Rect^l_{O_1, X_2}(\x, \y)$ be the set of long rectangles from $\x$ to $\y$ whose projection to $\grid'$ has multiplicity one at $O_1$ and at $X_2$. Now define
\begin{align*}
\hoxmb(\x) &= \sum_{\y\in N(\grid')} \sum_{\substack{r\in\Rect(\x,\y)\\O_1 \in r, \, \Int(r)\cap \mathbb{X}=X_2}} V_{2}^{O_{2}(r)}\cdots V_{n}^{O_{n}(r)}v^{\#(\Int(r)\cap\x)}\cdot \y\\
& + \sum_{\y\in N(\grid)} \sum_{r\in\Rect^l_{O_1, X_2}(\x,\y)} v \cdot \y.
\end{align*}

That is,  $\hoxmb$ counts embedded rectangles containing $O_1$, $X_2$, and possibly other $O$ markings, as well as long rectangles of width one with multiplicity one at $O_1$ and $X_2$. 

First, we show  that
\begin{equation}
\label{eqn:stab-inv}
\homb \circ \hxnb = \Id_I.
\end{equation}

Suppose $r_1 \ast r_2$ is a juxtaposition counted by $\homb \circ \hxnb$. Note that $c$ is a component of the starting generator for $r_1$ as well as a component of the ending generator for $r_2$. But $c$ is also a moving corner for both rectangles, so  $r_1$ and $r_2$ must share two corners on the circle $\alpha_1$. Thus, the domain for $r_1 \ast r_2$ must be a thin annulus, so $r_1$ takes the starting generator $\x \in I$ to some $\y \in N$, and $r_2$ takes $\y \in N$ back to $\x$. This proves \fullref{eqn:stab-inv}.

We now show that
\begin{equation}
\label{eqn:stab-homot}
\hxnb \circ \homb = \Id_N + \partial_N^N \circ \hoxmb + \hoxmb \circ \partial_N^N.
\end{equation}
Suppose a domain decomposes as $r_1 \ast r_2$, where $r_1$ contributes to the map $\homb$ and $r_2$ to $\hxnb$. We will show that this domain either has exactly one alternate decomposition into a pair of rectangles contributing to the second or third term on the right hand side of \fullref{eqn:stab-homot}, or is a thin annulus containing $O_1$ and $X_2$ and this contributes to the map $\Id_N$.

There are five types of domains that can be decomposed as juxtapositions $r_1 \ast r_2$ which contribute to $\hxnb \circ \homb$; see \fullref{fig:stab-inv}. 
\begin{figure}[ht]
    \labellist
        \pinlabel \small{$O_1$} at 37 136
        \pinlabel \small{$X_2$} at 37 120
        \pinlabel \small{$X_1$} at 51 136
        \pinlabel \small{$O_1$} at 37 57
        \pinlabel \small{$X_2$} at 37 41
        \pinlabel \small{$X_1$} at 51 57
        \pinlabel \small{$O_1$} at 121 136
        \pinlabel \small{$X_2$} at 121 120
        \pinlabel \small{$X_1$} at 135 136
        \pinlabel \small{$O_1$} at 121 57
        \pinlabel \small{$X_2$} at 121 41
        \pinlabel \small{$X_1$} at 135 57
        \pinlabel \small{$O_1$} at 179 122
        \pinlabel \small{$X_2$} at 179 106
        \pinlabel \small{$X_1$} at 193 122
        \pinlabel \small{$O_1$} at 268 128
        \pinlabel \small{$X_2$} at 268 112
        \pinlabel \small{$X_1$} at 282 128
        \pinlabel \small{$O_1$} at 268 32
        \pinlabel \small{$X_2$} at 268 16
        \pinlabel \small{$X_1$} at 282 32
        \pinlabel \small{$O_1$} at 352 156
        \pinlabel \small{$X_2$} at 352 140
        \pinlabel \small{$X_1$} at 366 156
        \pinlabel \small{$O_1$} at 352 58
        \pinlabel \small{$X_2$} at 352 42
        \pinlabel \small{$X_1$} at 366 58
    \endlabellist
    \includegraphics[scale=1]{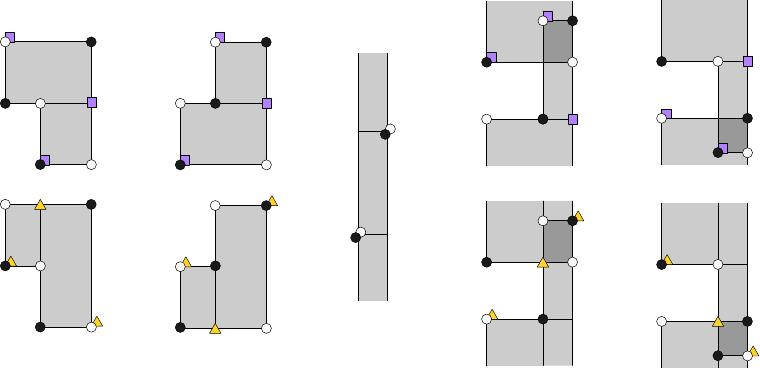}
    \caption{The five types of juxtapositions of rectangles that contribute to $\hxnb \circ \homb$. In each case, the starting grid state is given by the black circles and the terminal grid state by the white circles. Juxtapositions of each type except the third cancel out in pairs; intermediate generators for canceling pairs are depicted by purple squares and yellow triangles.}
    \label{fig:stab-inv}
\end{figure}
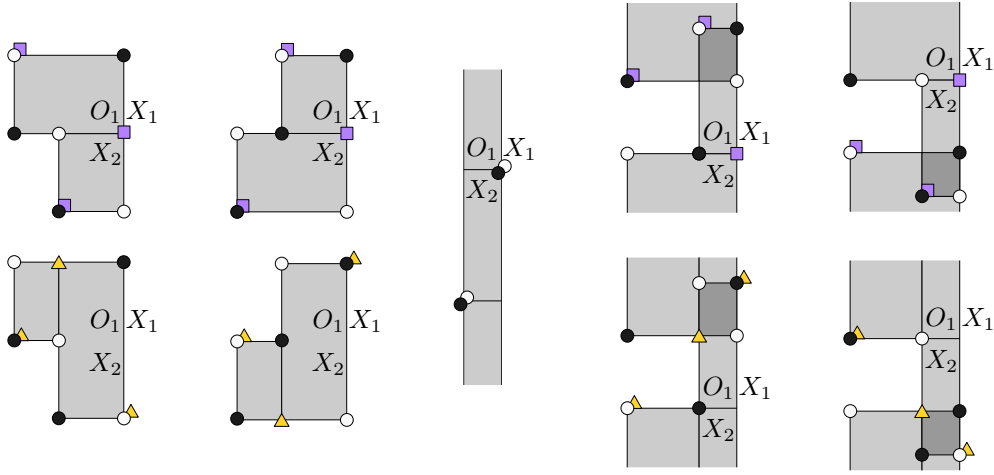
In the first two cases, shown in the first two columns of \fullref{fig:stab-inv}, the alternate decomposition is counted by $\partial_N^N \circ \hoxmb$ (first column) or $\hoxmb\circ\partial_N^N$ (second column). In the third case, shown in the third column of \fullref{fig:stab-inv},  the domain for $r_1 \ast r_2$ is the thin annulus containing $O_1$ and $X_2$, so $r_1$ takes the starting generator $\x \in N$ to some $\y \in I$, and $r_2$ takes $\y \in I$ back to $\x$, canceling the term of $\Id_N(\x)$.  Finally, the new cases where the interior of the domain contains a double point are also canceled by $\hoxmb\circ\partial_N^N$ and $\partial_N^N \circ \hoxmb$, respectively, where the $\hoxmb$ term is a long rectangle, as shown in the last two columns of \fullref{fig:stab-inv}. 
Other terms of $\hoxmb \circ \partial_N^N$ and $\partial_N^N \circ \hoxmb$ cancel in pairs.

This completes the proof that $\hxnb$ is a chain homotopy equivalence. 

Last, we show that the square in \fullref{diag:X:SE} commutes, which is equivalent to showing that $\bdy_N^I\circ \hxnb = V_1 - V_2$.\footnote{Note that $V_2 = -V_2$ in our ring; we write the minus sign simply since that would be the correct expression when working with $\Z$ coefficients.} 
Let $\x\in I$ and suppose $r_1 \ast r_2$ is a juxtaposition counted by $\bdy_N^I\circ \hxnb (\x)$. Note that $c$ is an initial corner for $r_1$ and a terminal corner for $r_2$, so either $r_1$ and $r_2$ share two corners on $\alpha_1$, or they share two corners on $\beta_1$. In the first case, the  domain of $r_1 \ast r_2$ is the thin vertical annulus containing $O_1$, $r_1$ is the unique width-one rectangle that starts at $\x$ and contains $X_2$, and $r_2$ is the unique width-one rectangle that starts at the the terminal generator of $r_1$ and contains $O_1$; the contribution of $r_1 \ast r_2$ to $\bdy_N^I\circ \hxnb (\x)$ is $V_1\x$. In the second case, the domain of $r_1 \ast r_2$ is the thin horizontal annulus containing $O_2$, $r_1$ is the unique height-one rectangle that starts at $\x$ and contains $X_2$, and $r_2$ is the unique height-one rectangle that starts at the the terminal generator of $r_1$ and contains $O_2$; the contribution of $r_1 \ast r_2$ to $\bdy_N^I\circ \hxnb (\x)$ is $V_2\x$. Thus, $\bdy_N^I\circ \hxnb (\x) = (V_1 - V_2)\x$.

Stabilizations of type \textit{X:NW} follow directly from this case by rotating each diagram by 180 degrees. Destabilizations follow since the maps are quasi-isomorphisms. 

This completes the proof of \ref{itm:s1}.
\\

The proof of Part \ref{itm:s2} is similar. This time, the argument relies on a commutative diagram 
  \[
    \xymatrix{
      I \ar[dd]_{e} \ar[rr]^{\bdy_I^N}
      & & N \ar[dd]^{e\circ \hxib}\\
      & & \\
      \GCmb(\grid)[V_1]\left\llbracket 1, 1 \right\rrbracket \ar[rr]^{V_1-V_2} 
      & & \GCmb(\grid)[V_1] 
    }
  \]
where $e\colon  I \to \GCmb(\grid)$ is the inverse of the map $e'$ defined earlier (i.e., the isomorphism induced by the natural identification of $\SG$ and $I$ in the opposite direction) 
and $\hxib$ is once again a map that counts rectangles whose interior intersects $\XX$ in $X_2$, this time from $N$ to $I$.

The proof that $\hxib$ is a chain homotopy equivalence and the diagram commutes is analogous to that of \ref{itm:s1}. Specifically, the domain decomposition analysis is identical to that in \ref{itm:s1} with the diagram reflected about the vertical line through $c$.

Thus, we have the desired isomorphism on homology.

Last, we verify that this isomorphism acts as claimed on $\lpb(\grid')$ and $\lmb(\grid')$.

This time, the canonical generators $\xp(\grid')$ and $\xm(\grid')$ are elements of $N$. The only rectangle starting from $\xp(\grid')$ counted by $\hxib$ is the complement of the square that contains $O_1$ in the corresponding thin vertical annulus; note that the only marking in the interior of this rectangle is $X_2$. The terminal grid state $\y$ for this rectangle is sent to $\xp(\grid)$ via the map $e$. Similarly to  \ref{itm:s1}, this shows that the isomorphism on homology sends $\lpb(\grid')$ to $\lpb(\grid)$. 

Similarly, the only rectangle starting from $\xm(\grid')$ counted by $\hxib$ is the complement of the square that contains $O_1$ in the corresponding thin horizontal annulus. But this rectangle contains an $O$ marking, so $\hxib(\xm(\grid'))= U\cdot \y$, where $\y$ is the terminal grid state for the rectangle. Once again $e$ sends $\y$ to $\xm(\grid)$, so we see that the isomorphism on homology sends $\lmb(\grid')$ to $U\cdot \lmb(\grid)$. 
\\

Part~\ref{itm:s3} follows similarly. 
\end{proof}

\subsection{The proof of invariance completed}

\fullref{lem:comm} and \fullref{lem:stab} now easily imply the main results in  \fullref{ssec:intro-invts}. First, we show that $\xpm$ give rise to Legendrian link invariants $\lpmb$ and $\lhpmb$.

\begin{proof}[Proof of \fullref{thm:legm}]
If the toroidal grid diagrams $\grid$ and $\grid'$ represent the same Legendrian link, then by \cite[Proposition 12.2.6]{OSS15} there is sequence of commutations and (de)stabilizations of type \textit{X:SE} and \textit{X:NW} that transforms $\grid$ to $\grid'$. The second part of \fullref{lem:comm} shows that the homology classes of $\xpm$ are identified by the isomorphisms on homology induced by commutations. Part~\ref{itm:s1} of \fullref{lem:stab} shows that the homology classes are identified by the isomorphisms on homology induced by (de)stabilizations. This completes the proof of invariance.
\end{proof}

\begin{proof}[Proof of \fullref{thm:legh}]
The corollary follows directly from \fullref{thm:legm}, by specializing to the quotient $\GHhb$.
\end{proof}

\begin{proof}[Proof of \fullref{prop:tilde-hat}]
This follows for the same reasons as for the non-enhanced invariants. Namely, the projection $\GChb(\grid)\to \GCtb(\grid)$ induces an injection on homology. But this projection carries $\xpm$ to $\xpm$, so the result follows. 
\end{proof}

Next, we show that we also get a transverse link invariant $\thetab$: 

\begin{proof}[Proof of \fullref{thm:transverse}]
Since transverse links are just Legendrian links up to negative stabilization, and negative stabilizations can be realized by grid stabilizations of type \textit{X:SW}, the result follows from \fullref{thm:legm} and Part~\ref{itm:s2} of \fullref{lem:stab} in the case of $\xp$.
\end{proof}

More generally, we see that $\lpb$ and $\lmb$ behave like $\lp$ and $\lm$ under stabilization:

\begin{proof}[Proof of \fullref{thm:leg-stab}]
This follows from \fullref{thm:legm} and Parts~\ref{itm:s2} and \ref{itm:s3} of \fullref{lem:stab}.
\end{proof}

\begin{proof}[Proof of \fullref{cor:stab-vanishing}]
This follows immediately from \fullref{thm:leg-stab}, by specializing to  $\GHhb$.
\end{proof}

\section{The double-point enhanced GRID invariant: Obstructing decomposable Lagrangian cobordisms}\label{sec:cob}

As in \cite{JPSWW22}, we will prove \fullref{thm:cob} by defining analogous filtered chain maps corresponding to pinch and birth moves, so that the induced maps on spectral sequences preserve the double-point enhanced GRID classes. Recall once again that we don't know whether the canonical elements induce GRID invariants in the {\bf filtered} double-point enhanced grid homology, or even whether the filtered double-point enhanced grid homology is a link invariant. So at present, we only care that the maps our filtered maps induce on associated graded objects imply there is an obstruction in the double-point enhanced {\bf fully blocked} setting. A proof of stabilization and destabilization invariance in the double-point enhanced filtered setting would immediately upgrade our result to a filtered one, a double-point enhanced analogue of \cite{JPSWW22}.

\subsection{Pinches} 
\label{sec:pinch}

Suppose $\Legp$ is a Legendrian link obtained from a Legendrian link $\Legm$ by a pinch move. Then there exist grid diagrams $\Gp$ and $\Gm$ representing $\Legp$ and $\Legm$, respectively, such that they differ only in the placement of one pair of $X$ markings or one pair of $O$ markings in adjacent rows, as in \fullref{fig:sep}. If the diagrams differ at their $X$ markings, we say $\Gp$ is obtained from $\Gm$ by an $X$ swap, otherwise we say $\Gp$ is obtained from $\Gm$ by an $O$ swap. It can be guaranteed that the two swapped markings are separated by at least two vertical lines.
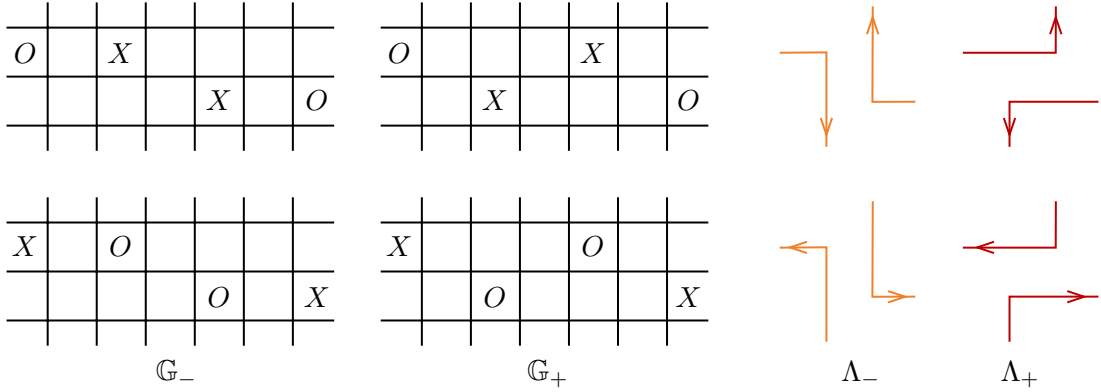
\begin{figure}[ht]
        \labellist
    \pinlabel $O$ at 7 111
    \pinlabel $X$ at 43 111
    \pinlabel $X$ at 80 93
    \pinlabel $O$ at 117 93
    \pinlabel $X$ at 7 38
    \pinlabel $O$ at 43 38
    \pinlabel $O$ at 80 19
    \pinlabel $X$ at 117 19
    \pinlabel $O$ at 148 111
    \pinlabel $X$ at 221 111
    \pinlabel $X$ at 184 93
    \pinlabel $O$ at 257 93
    \pinlabel $X$ at 148 38
    \pinlabel $O$ at 221 38
    \pinlabel $O$ at 184 19
    \pinlabel $X$ at 257 19
    \pinlabel $\G_-$ at 64 -10
    \pinlabel $\G_+$ at 204 -10
    \pinlabel $\Lambda_-$ at 322 -10
    \pinlabel $\Lambda_+$ at 382 -10
    \endlabellist
     \includegraphics[scale=1]{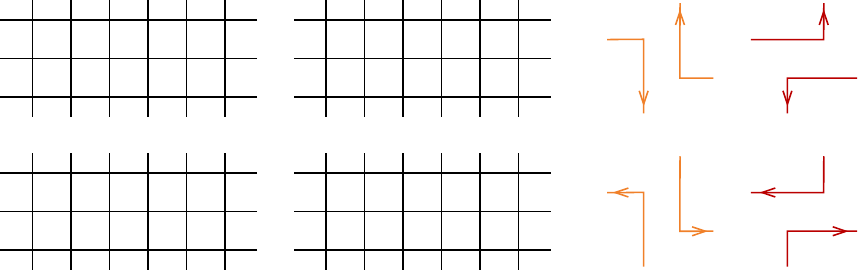}
    \vspace{0.3cm}
  \caption{The grid diagrams $\G_\pm$ corresponding to the two types of pinch moves. Top: An $X$ swap. Bottom: An $O$ swap.}
    \label{fig:sep}
    \end{figure}

We combine $\Gp$ and $\Gm$ into a single \emph{combined diagram}, as in \fullref{fig:combo}. We label the curves separating the swapped markings by $\alpha$ and $\alpha'$, so that $\alpha$ corresponds to $\Gp$ and $\alpha'$ corresponds to $\Gm$, and we label the two intersection points of $\alpha$ and $\alpha'$ by $a$ and $b$, as seen in \fullref{fig:combo}. We will use $a$ and $b$ to define the maps in \fullref{prop:x-swap-filt} and \fullref{prop:o-swap-filt}.

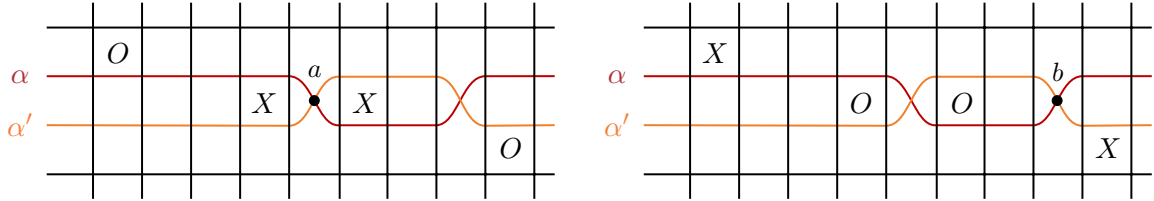
\begin{figure}[ht]
    \labellist
    \pinlabel $O$ at 27 55
    \pinlabel $X$ at 82 36
    \pinlabel $X$ at 120 36
    \pinlabel $O$ at 175 19
    \pinlabel {\small{$a$}} at 101 48
    \pinlabel $X$ at 252 55
    \pinlabel $O$ at 307 36
    \pinlabel $O$ at 345 36
    \pinlabel $X$ at 400 19
    \pinlabel {\small{$b$}} at 381 48
    \pinlabel \textcolor{Maroon}{$\alpha$} at -10 46
    \pinlabel \textcolor{Orange}{$\alpha'$} at -10 28
    \pinlabel \textcolor{Maroon}{$\alpha$} at 215 46
    \pinlabel \textcolor{Orange}{$\alpha'$} at 215 28
    \endlabellist
  \includegraphics[scale=1]{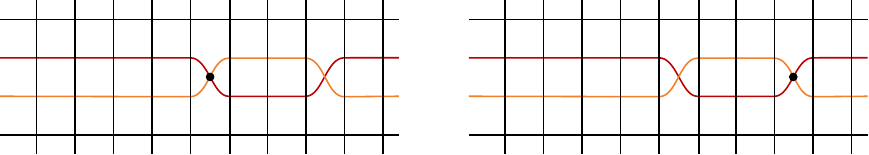}
  \caption{Left: The combined diagram corresponding to an $X$ swap. Right: The combined  diagram corresponding to an $O$ swap.}
\label{fig:combo}
\end{figure}

\begin{lemma}
\label{prop:x-swap-filt}
Suppose that $\Gp$ is obtained from $\Gm$ by an X swap. Then there exists a filtered chain homomorphism
\[\xswapb: \fGCtb(\Gp) \to \fGCtb(\Gm)\left\llbracket 1, \frac{\left|\Legp\right| - \left|\Legm\right| +1}{2} \right\rrbracket\]
such that
\[\xswapb(\xpm(\Gp)) = \xpm(\Gm) + \y,\]
where
\[\y \in \mathcal F_{A(\xpm(\Gm))-1}\fGCtb(\Gm).\]
\end{lemma}

\begin{proof}
We proceed similarly as in \cite[Lemma 4.1]{JPSWW22}, this time allowing the chain map to also count pentagons with components of the starting generator in their interior.

Given $\x \in \SGp$ and $\y \in \SGm$, define sets of pentagons $\Pent(\x,\y)$ and $\Pent_{\OO}(\x,\y)$ as in \fullref{def:pent}; note that the definition of a pentagon implies that $a$ is one of its five vertices.

 Let
\[\xswapb\colon \fGCtb(\G_+) \to \fGCtb(\G_-)\]
be the linear map defined by
\[\xswapb(\mathbf{x}) = \sum_{\y\in\boldsymbol{S}(\mathbb G_-)}\ \sum_{p\in\Pent_{\OO}(\x,\y)} v^{\#(\Int(p)\cap\x)}\cdot \y.\]

The proof that $\xswapb$ is a chain map is similar to \cite[Lemma 4.1]{JPSWW22}. The primary difference is that the domains we consider may contain components of the starting generator in their interior. The domain analysis is analogous to that in \fullref{sec:comm}, except that here we do not have the cases of thin annuli or long pentagons, since all thin annuli on the combined diagram contain  $O$ markings. 

Now, we look at Maslov grading and Alexander filtration. Let $p \in \Pent(\x,\y)$ be a pentagon from $\x$ to $\y$ with $\Int(p) \intersect \x = \emptyset$. In \cite[Lemma 4.1]{JPSWW22}, it is shown that
\[M(\y)-M(\x) = -1,\]
and
\[A(\y)-A(\x) = \#(\Int(p) \cap \mathbb X) + \frac{\left|\Legm\right| - \left|\Legp\right| -1}{2}.\]
Now suppose that $p\in \Pent(\x,\y)$ has $k$ interior double points. Allowing interior double points does not change the relative Alexander grading formula, so 
\[A(v^k\y)-A(\x) = A(\y)-A(\x) = \#(\Int(p) \cap \mathbb X) + \frac{\left|\Legm\right| - \left|\Legp\right| -1}{2}.\] 
However, $\countJ(\y,\y) - \countJ(\x,\x)$ decreases by $2k$ since we add $k$ to account for the double points above and to the right of the bottom left moving component of $\x$, and we add $k$ again to account for the double points below and to the left of the top right moving component of $\x$. So $M(\y)- M(\x) = -1 -2k$.
Since $M(v^k \y) = M(\y)+ 2k$, we still have that $M(v^k \y)-M(\x)=-1$. 

Last, we show that the map $\xswapb$ sends $\xpm(\Gp)$ to $\xpm(\Gm)+\y$ where $\y\in \mathcal F_{A(\xpm(\Gm))-1}\fGCtb(\Gm)$. Similar to \cite[Lemma 3.3]{BLW22}, there is a unique pentagon with no $X$ markings in its interior starting at $\xpm(\Gp)$; this pentagon ends at $\xpm(\Gm)$. Since all other pentagons starting at $\xpm(\Gp)$ contain $X$ markings, their ending generators are in a strictly lower filtration level than $A(\xpm(\Gm))$; all pentagons starting at $\xpm(\Gp)$ with double points in their interior are of this second type.
\end{proof}

\begin{lemma}
\label{prop:o-swap-filt}
Suppose that $\Gp$ is obtained from $\Gm$ by an O swap. Then there exists a filtered chain homomorphism
\[\oswapb: \fGCtb(\Gp) \to \fGCtb(\Gm)\left\llbracket 1, \frac{\left|\Legp\right| - \left|\Legm\right| +1}{2} \right\rrbracket\]
such that
\[\oswapb(\xpm(\Gp) = \xpm(\Gm).\]
\end{lemma}

\begin{proof}
Let $\x \in \SGp$ and $\y \in \SGm$. As in \cite{JPSWW22}, let $\Tri(\x,\y)$ be the set of embedded triangles in the combined diagram with non-reflex angles and vertices the points $(\x \triangle \y) \cup \{b\}$. More precisely, If $p \in \Tri(\x,\y)$, then the boundary of $p$ consists of an arc of $\alpha$, an arc of $\alpha'$, and an arc on a vertical circle. Note that with order induced by the boundary orientation on $p$, we encounter $\x, b, \y$ in this order, and that $\Int(p) \cap \x = \emptyset$. Also note that if $\x$ and $\y$ do not agree at exactly $n-1$ points, then $\Tri(\x,\y)$ is empty. 
 Let $\Tri_{\OO}(\x,\y)$ be the subset of $\Tri(\x,\y)$ such that if $p \in \Tri_{\OO}(\x,\y)$, then $\Int(p) \cap \OO = \emptyset$.  We proceed similarly to \cite[Lemma 4.2]{JPSWW22}. Let 
\[\oswapb: \fGCtb(\Gp) \to \fGCtb(\Gm)\] 
be the linear map defined on generators by
\[\oswapb(\x) = \sum_{\y\in\SGm}\ \sum_{p\in \Tri_{\OO}(\x,\y)} \y.\]
This map is a chain map since domains that arise from juxtapositions of triangles and rectangles can be decomposed in exactly two ways as such, just as in \cite[Lemma 4.2]{JPSWW22}; see also \cite[Lemma 3.4]{Won17} for details. This time, however, domains may contain components of the initial grid state in their interior, so there are a couple of new geometric configurations. We carry out the case analysis below.

Suppose $p_1\ast r_1$ is a juxtaposition of a triangle and a rectangle. There are two cases: either the triangle and rectangle share no corners, or they share one corner. 

In the former case, the domain has a unique alternate decomposition, as $r_2 \ast p_2$, where $r_1$ and $r_2$ have the same support, and so do $p_1$ and $p_2$; see the top row of \fullref{fig:tri-big}. The terminal grid state for the two juxtapositions is the same. In this case, either each rectangle contains the entire vertical edge of the respective triangle in its interior, or neither rectangle intersects the vertical edge of the respective triangle. This, $r_1$ and $r_2$ have equal contribution to the power of $v$ in $\oswapxb \circ \bdtb (\x) + \bdtb\circ \oswapxb(\x)$, and hence so do the two juxtapositions. 

In the latter case, the domain has one reflex corner; cutting in the other way at this corner results in the unique alternate decomposition as $r_2 \ast p_2$, with the same terminal grid state as $p_1\ast r_1$; see the bottom row of \fullref{fig:tri-big}. Once again the two rectangles contribute equal powers of $v$ to $\oswapxb \circ \bdtb (\x) + \bdtb\circ \oswapxb(\x)$.

\begin{figure}[H]
    \includegraphics{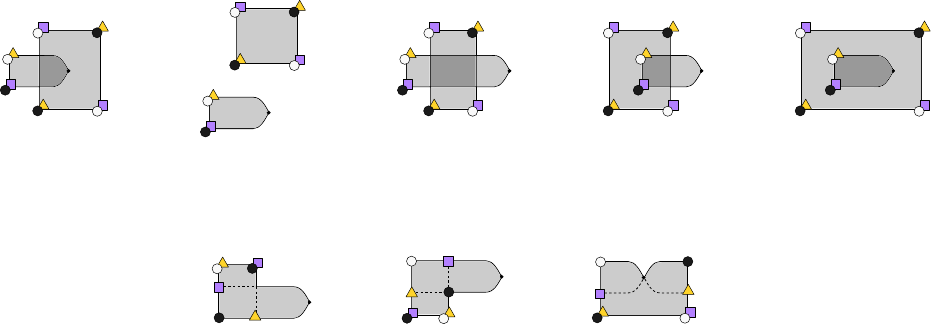}
    \caption{The types of domains that contribute to $\oswapxb \circ \bdtb + \bdtb\circ \oswapxb$. Each domain can be decomposed in two ways that contribute to $\oswapxb \circ \bdtb + \bdtb\circ \oswapxb$, with intermediate grid state depicted by purple squares or yellow triangles, respectively.}
    \label{fig:tri-big}
\end{figure}

Thus, $\oswapxb \circ \bdtb (\x) + \bdtb\circ \oswapxb(\x) = 0$.

Since the triangles we count are the same as those in \cite[Lemma 4.2]{JPSWW22}, we again have that $\oswapxb$ respects the Maslov grading and Alexander filtration, and that $\oswapb(\xpm(\Gp) = \xpm(\Gm)$.
\end{proof}

\subsection{Birth Moves}

\begin{lemma}\label{lem:birth}
Suppose that $\Gp$ is obtained from $\Gm$ by a birth move, with the birth occurring directly to the bottom right of an $O$, as in \fullref{fig:BirthDiag}. Then there exists a filtered chain homomorphism
\[\birthmapb: \fGCtb(\Gp) \to \fGCtb(\Gm)\llbracket -1,0 \rrbracket\]
such that
\[\birthmapb(\xpm(\Gp)) = \xpm(\Gm).\]
\end{lemma}
    
\begin{proof}
The strategy of this proof will be to extend the birth maps in \cite[Proposition 3.9]{BLW22} and \cite[Lemma 4.3]{JPSWW22} to allow rectangles whose interior contains $X$ markings, components of the initial grid state, or both. We include the full set up here for the sake of completeness.

Define the points $\bira$ and $\birb$ of $\Gp$ as in \fullref{fig:BirthDiag}. 
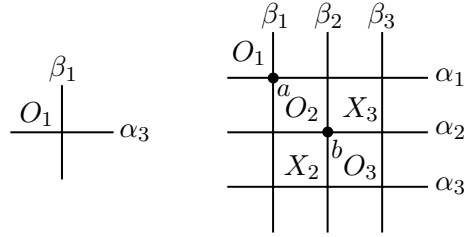
\begin{figure}[ht]
  \vspace{.2cm}
  \labellist
  \pinlabel $O_1$ at 10 45
  \pinlabel $\beta_1$ at 20 62
  \pinlabel $\alpha_3$ at 47 38
  \pinlabel $\beta_1$ at 100 82
  \pinlabel $\beta_2$ at 120 82
  \pinlabel $\beta_3$ at 140 82
  \pinlabel $\alpha_1$ at 166 59
  \pinlabel $\alpha_2$ at 166 39
  \pinlabel $\alpha_3$ at 166 18
  \pinlabel $O_1$ at 90 68
  \pinlabel $O_2$ at 110 47
  \pinlabel $O_3$ at 132 25
  \pinlabel $X_2$ at 110 25
  \pinlabel $X_3$ at 132 47
  \pinlabel {\small{$a$}} at 103 54
  \pinlabel {\small{$b$}} at 123 32
  \endlabellist
  \includegraphics[scale=1]{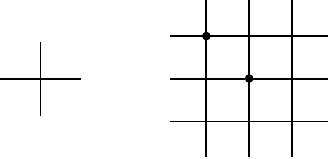}
  \caption{Left: Part of the grid diagram $\G_-$. Right: The corresponding part of the diagram $\G_+$ obtained from $\G_-$ by a birth move.}
  \label{fig:BirthDiag}
\end{figure}
Then, the generating set $\SGp$ can be expressed as the disjoint union
\[\SGp = \AB \disjunion \AN \disjunion \NB \disjunion \NN\]
where
\begin{itemize}
    \item $\AB$ consists of $\x \in \SGp$ with $\bira,\birb \in \x$;
    \item $\AN$ consists of $\x \in \SGp$ with $\bira \in \x$ and $\birb \notin \x$;
    \item $\NB$ consists of $\x \in \SGp$ with $\bira \notin \x$ and $\birb \in \x$; and
    \item $\NN$ consists of $\x \in \SGp$ with $\bira,\birb \notin \x$.
\end{itemize}

This decomposition of the generating set induces a decomposition of the vector space $\fGCtbGp$ as a direct sum
\[\fGCtbGp = \ABt \dirsum \ANt \dirsum \NBt \dirsum \NNt\]
where the summands are the subspaces generated by the corresponding subsets of $\SGp$. Note, we have a sequence of subcomplexes
\[\NNt \subset \NBt \dirsum \NNt \subset \ANt \dirsum \NBt \dirsum \NNt \subset \fGCtbGp,\]
just as in \cite{JPSWW22}.
This follows from the observation that any rectangle terminating at either $\bira$ or $\birb$ must necessarily contain an $O$, and thus is not counted by the differential. Let 
\[\ABcomplex\]
be the quotient complex of $\fGCtbGp$ by $\ANt \dirsum \NBt \dirsum \NNt$. There is a natural bijection of the generators in $\AB$ and those in $\SGm$ given by
\[\x \mapsto \xpr \coloneq \x \setminus \set{\bira,\birb}.\]
Note that there is further a bijection between rectangles that contain no $O$ markings between grid states in $\AB$ and rectangles that contain no $O$ markings in $\Gm$, since any rectangle in $\Gp$ from a grid state in $\AB$ to another grid state in $\AB$ which contains in its interior the $2 \cross 2$ block corresponding to the birth move must necessarily contain $O_1$ as well.
Thus, the bijection on grid states extends linearly to an isomorphism
\[\birthemap: \ABcomplex \to (\fGCtbGm, \bdtb)\]
of chain complexes, since for $\x,\y \in \AB$ there is also a bijection
\[\Rect_{\OO}(\x,\y) \to \Rect_{\OO}(\x',\y').\]

Next, we see that the bijection respects the Maslov--Alexander bigrading. Fix $\x\in \AB$. Let $x_1$ and $o_1$ be the number of components of $\x$ and the number of $O$ markings, respectively, to the bottom left of the $2 \cross 2$ block corresponding to the birth move. Let $x_2$ and $o_2$ be the number of components of $\x$ and the number of $O$ markings, respectively, to the top right of the $2 \cross 2$ block. Observe that $\J(\x,\x) = \J(\x',\x') + 2x_1 + 2x_2$, $\J(\x,\OO) = \J(\x',\OO) + x_1 + x_2 + o_1+o_2$, and $\J(\OO,\OO)$ increases by $2o_1+2o_2$ as we move from $\Gm$ to $\Gp$, so $M_{\OO}(\x) = M_{\OO}(\x')$. Similarly, $M_{\XX}(\x) = M_{\XX}(\x')+1$, since the pair $(X_2,X_3)$ contributes an additional count to $\J(\XX, \XX)$ in $\Gp$. Thus, $e$ respects the bigrading. 

Let $\OO$ and $\XX$ denote the sets of $O$ markings and $X$ markings in $\Gp$, respectively. Given $\x \in \NB$ and $\y \in \AB$, we define
\[\RectAB{\x}{\y} \subset \RectGp{\x}{\y}\]
as the subset of rectangles $p$ satisfying
\begin{itemize}
    \item $p \intersect \paren{\OO \union \set{O_2,O_3}} = \set{O_2,O_3}$
    \item $p \intersect \paren{\XX \union \set{X_2,X_3}} \supseteq \set{X_2,X_3}$
    \item $\Int(p) \intersect \x = \Int(p) \intersect \y \supseteq \set{b}.$
\end{itemize}
The third item on this list is what differs from \cite{JPSWW22}, where equality instead of containment was required.

Let
\[\birthpsimap: \NBt \to \ABt \]
be the linear map defined on generators by counting these rectangles as follows:
\[\birthpsimap(\x) = \sum_{\y \in \AB} \sum_{p \in \RectAB{\x}{\y}} v^{\#(\Int(p) \intersect \paren{\x \setminus \set{b}})}\y.\]

Let
\[\birthpimap: \fGCtbGp \to \NBt\]
be the projection onto the summand $\NBt$.

Lastly, define
\[\birthmapb: \fGCtbGp \to \fGCtbGm\]
as the linear map given by the composition
\[\birthmapb =\birthemap \comp \birthpsimap \comp \birthpimap.\]

With the map $\birthmapb$ now defined, we will first show that
\[\birthmapb(\xpm(\Gp)) = \xpm(\Gm).\]
Note, $\xpm(\Gp) \in \NB$, so
\[\birthpimap(\xpm(\Gp)) = \xpm(\Gp).\]
As seen in \fullref{fig:BX}, there is a unique rectangle contributing to $\birthpsimap(\xp(\Gp))$ and a unique rectangle contributing to $\birthpsimap(\xm(\Gp))$. Since rectangles $p \in \RectAB{\x}{\y}$ must terminate at a grid state $\y$ with $a \in \y$, $a$ must be the northwest corner of $p$. The northern and western edges of $p$ must then extend until they reach the unique components of $\x$ which lie on $\al_1$ and $\be_1$. \fullref{fig:BX} also shows that composing with $\birthemap$ then yields 
\[\birthemap(\birthpsimap(\xpm(\Gp))) = \xpm(\Gm).\]
Hence,
\[\birthmapb(\xpm(\Gp)) = \xpm(\Gm).\]

\begin{figure}[ht]
        \labellist
    \pinlabel $O_1$ at 6 71
    \pinlabel $O_2$ at 24 53
    \pinlabel $O_3$ at 43 34
    \pinlabel $X_2$ at 24 34
    \pinlabel $X_3$ at 43 53
    \pinlabel $X_1$ at 6 -3
    \pinlabel $X_1$ at 145 18
    \pinlabel $O_1$ at 145 51
    \pinlabel $O_1$ at 225 63
    \pinlabel $O_3$ at 262 27
    \pinlabel $O_2$ at 243 45
    \pinlabel $X_2$ at 262 45
    \pinlabel $X_3$ at 243 27
    \pinlabel $X_1$ at 298 63
    \pinlabel $O_1$ at 382 45
    \pinlabel $X_1$ at 414 45
    \pinlabel $e$ at 96 45
    \pinlabel $e$ at 335 45
    \pinlabel $\grid_+$ at 33 -15
    \pinlabel $\grid_-$ at 150 -15
    \pinlabel $\grid_+$ at 262 -15
    \pinlabel $\grid_-$ at 401 -15
    \pinlabel $\vdots$ at 58 18
    \pinlabel $\vdots$ at 162 39
    \pinlabel $\dots$ at 280 9
    \pinlabel $\dots$ at 399 26
    \endlabellist
    \vspace{.6cm}
  \includegraphics[scale=1]{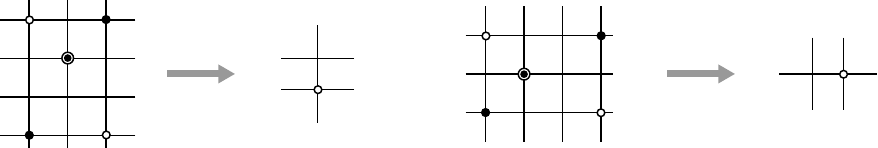}
  \vspace{.6cm}
  \caption{Left: In $\grid_+$, the generator $\xp (\grid_+)$ (black circles) and the generator $\y = \birthpsimap (\xp (\grid_+))$ (white circles). The corresponding generator $\birthmapb (\xp (\grid_+)) = \xp (\grid_-)$ is shown in $\grid_-$. Right: The analogous illustration for $\xm$.}
  \label{fig:BX}
\end{figure}

Next, we will show that $\birthmapb = \birthemap \comp \birthpsimap \comp \birthpimap$ is a chain map. That is, we need to show
\[\bdtb \comp (\birthemap \comp \birthpsimap \comp \birthpimap) + (\birthemap \comp \birthpsimap \comp \birthpimap) \comp \bdtb = 0.\]
Note, both left hand side terms vanish on $\x \in \AN \union \NN$. The first term vanishes because
\[\birthpimap(\x) = 0\]
and the second because
\[\birthpimap \comp \bdtb(\x) = 0.\]
This is because $\x$ cannot be taken to some $\y$ with $b \in \y$ by $\bdtb$ since $b$ has $O$ markings to its northwest and southeast. Hence, we can restrict our focus to $\x \in \AB \cup \NB$. Given $\y \in \SGm$, the coefficients of $\y$ in $(\bdtb \comp \birthmapb)(\x)$ and $(\birthmapb \comp \bdtb)(\x)$
count the juxtapositions of rectangles from $\x$ to $\y$ of the form $p \ast r$ and $r \ast p$, respectively, where $p$ are rectangles that contribute to $\birthpsimap$ and $r$ are rectangles that contribute to $\bdtb$. The cases where the interiors of $p$ and $r$ do not contain any components of $\x$ and $\y$ are covered in \cite[Lemma 3.10]{BLW22}. Additionally, the cases covered in \cite[Lemma 3.10]{BLW22} are geometrically similar to the cases here where the interiors of $p$ and $r$ only contain components of $\x$ and $\y$ which are not corners of $p$ or $r$. 
Last, suppose a domain from $\x$ to $\y$ decomposes as a juxtaposition of a rectangle $r$ that contributes to $\bdtb$ and a rectangle $p$ that contributes to $\birthmapb$ (regardless of order), and the interior of one rectangle contains a component of $\x$ and/or $\y$ which is a corner of the other rectangle. Then the domain has a unique alternate decomposition that contributes to $(\bdtb \comp \birthmapb)(\x) + (\birthmapb \comp \bdtb)(\x)$, where the rectangles has the same support as $r$ and $p$ but come in the opposite order, and the terminal state is $\y$ again.  The two juxtapositions contribute the same total power of $v$ to the coefficient of $\y$. See \fullref{fig:birth1} and \fullref{fig:birth2} for the possible cases of these types of domains. Hence, $\birthmapb$ is a chain map.
\begin{figure}[ht]
    \labellist
        \pinlabel \small{$O_2$} at 103 330
        \pinlabel \small{$X_2$} at 103 315
        \pinlabel \small{$X_3$} at 118 330
        \pinlabel \small{$O_3$} at 118 315
        \pinlabel \small{$O_2$} at 199 330
        \pinlabel \small{$X_2$} at 199 315
        \pinlabel \small{$X_3$} at 214 330
        \pinlabel \small{$O_3$} at 214 315
        \pinlabel \small{$O_2$} at 86 216
        \pinlabel \small{$X_2$} at 86 201
        \pinlabel \small{$X_3$} at 101 216
        \pinlabel \small{$O_3$} at 101 201
        \pinlabel \small{$O_2$} at 200 197
        \pinlabel \small{$X_2$} at 200 182
        \pinlabel \small{$X_3$} at 215 197
        \pinlabel \small{$O_3$} at 215 182
        \pinlabel \small{$O_2$} at 31 82
        \pinlabel \small{$X_2$} at 31 67
        \pinlabel \small{$X_3$} at 46 82
        \pinlabel \small{$O_3$} at 46 67
        \pinlabel \small{$O_2$} at 151 82
        \pinlabel \small{$X_2$} at 151 67
        \pinlabel \small{$X_3$} at 165 82
        \pinlabel \small{$O_3$} at 165 67
        \pinlabel \small{$O_2$} at 254 96
        \pinlabel \small{$X_2$} at 254 81
        \pinlabel \small{$X_3$} at 269 96
        \pinlabel \small{$O_3$} at 269 81
    \endlabellist
    \includegraphics[scale=1]{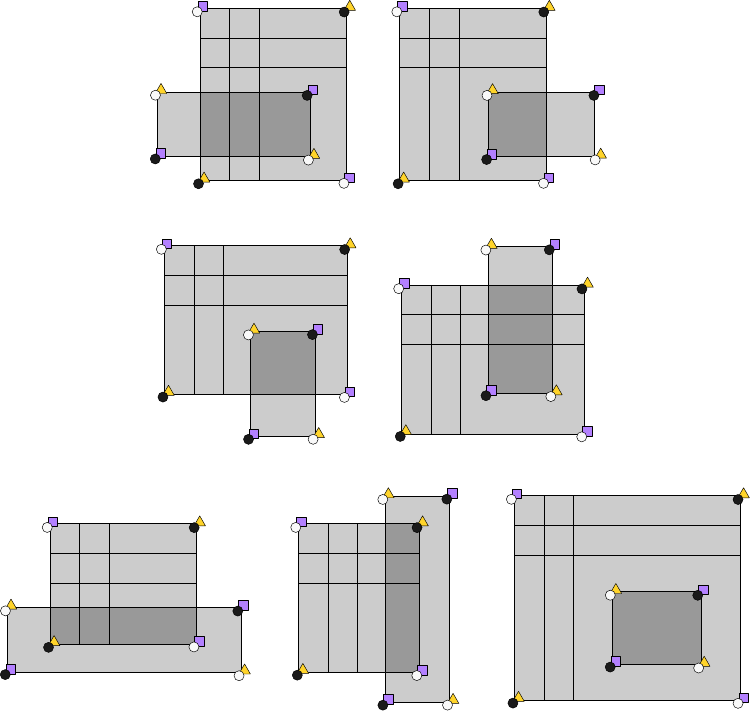}
    \caption{The types of domains that contribute to $\bdtb \comp \birthmapb+\birthmapb \comp \bdtb$ such that one rectangle contains a corner of the other in its interior: The case where one rectangle contains an edge of the other in its interior.}
\label{fig:birth1}
\end{figure}
\begin{figure}[ht]
    \labellist
        \pinlabel \small{$O_2$} at 32 75
        \pinlabel \small{$X_2$} at 32 90
        \pinlabel \small{$X_3$} at 47 75
        \pinlabel \small{$O_3$} at 47 90
        \pinlabel \small{$O_2$} at 139 75
        \pinlabel \small{$X_2$} at 139 90
        \pinlabel \small{$X_3$} at 154 75
        \pinlabel \small{$O_3$} at 154 90
        \pinlabel \small{$O_2$} at 270 53
        \pinlabel \small{$X_2$} at 270 68
        \pinlabel \small{$X_3$} at 285 53
        \pinlabel \small{$O_3$} at 285 68
    \endlabellist
    \includegraphics[scale=1]{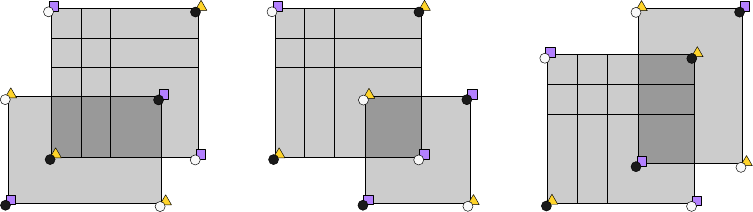}
    \caption{The types of domains that contribute to $\bdtb \comp \birthmapb+\birthmapb \comp \bdtb$ such that one rectangle contains a corner of the other in its interior: The case where one rectangle does not contain an edge of the other in its interior.}
    \label{fig:birth2}
\end{figure}

Note, the maps $\birthemap$ and $\birthpimap$ are both homogeneous with respect to the Maslov grading and the Alexander filtration. Now, we focus on the map $\birthpsimap$. Let $\x \in \NB$ and $\y \in \AB$ such that $\RectAB{\x}{\y} \neq \emptyset$. Let $p \in \RectAB{\x}{\y}$ be the rectangle connecting the two generators. Let $k = \#(\Int(p) \intersect (\x \intersect \y))$ and $\ell = \#(\Int(p) \intersect \XX)$. Then, \fullref{eq:maslov-gr-rel} and \fullref{eq:alex-gr-rel} imply that
\begin{align*}
M(\x)-M(\y) &= 2k - 1; \text{ and}\\
A(\x)-A(\y) &= \ell.
\end{align*}
Therefore,
\begin{align*}
M(\birthpsimap(\x)) &= M(v^k\y) = M(\y)+2k = M(\x)+1; \text{ and}\\
A(\birthpsimap(\x)) &= A(v^k\y) = A(\y) = A(\x)-\ell.
\end{align*}
It follows that $\birthmapb$ is a filtered chain map viewed as a map to $\fGCtbGm\llbracket -1,0 \rrbracket$.
\end{proof}

\subsection{Proof of weak functoriality and obstruction}

We now show that $\ltpb$ and $\ltmb$ satisfy a weak functoriality under decomposable Lagrangian cobordisms.

\begin{proof}[Proof of \fullref{thm:functoriality}]
First, note that the filtered maps from \fullref{prop:x-swap-filt}, \fullref{prop:o-swap-filt}, and \fullref{lem:birth} 
descend to bigraded (with appropriate degree shifts) homomorphisms on the associated graded objects $\GHtb$ that send $\ltpmb$ to $\ltpmb$.

Suppose that $\Gm$ and $\Gp$ are grid diagrams that represent Legendrian links $\Legm$ and $\Legp$, and suppose that there exists a decomposable Lagrangian cobordism $\Lag$ from $\Legm$ to $\Legp$. The $\Lag$ is isotopic to a composition of Legendrian isotopies, pinches, and births. It follows that $\Gp$ can be obtained from $\Gm$ by a sequence of commutations, (de)stabilizations of type \textit{X:SE} or \textit{X:NW}, $X$ swaps, an $O$ swaps, and birth moves of the type shown in \fullref{fig:BirthDiag}. \fullref{thm:legh}, \fullref{prop:x-swap-filt}, \fullref{prop:o-swap-filt}, and \fullref{lem:birth} then imply the existence of the desired homomorphism from $\GHtb(\Gp)$ to $\GHtb(\Gm)$. 
\end{proof}

\begin{proof}[Proof of \fullref{thm:cob}]
This is immediate from \fullref{thm:functoriality} and \fullref{prop:tilde-hat}.
\end{proof}

\begin{proof}[Proof of \fullref{cor:filling}]
A decomposable Lagrangian filling of $\Leg$ is a composition of a birth and a decomposable Lagrangian cobordism from the undestabilizable Legendrian unknot $\Leg_U$ to $\Leg$. Since $\GC$ $\lhpmb(\Leg_U) \neq 0$, the corollary follows from \fullref{thm:cob}.
\end{proof}

\section{Comparison with the GRID invariants}\label{sec:comp}

In this section, we compare the GRID invariants with the double-point enhanced GRID invariants and discuss some computations. First, we show that triviality in the double-point enhanced setting implies triviality in the non-enhanced setting:

\begin{proof}[Proof of \fullref{thm:compare}]
    Let $\Leg$ be a Legendrian link, and suppose $\lhpb(\Leg) = 0$. Let $\grid$ be a grid diagram for $\Leg$. Then $\ltpb(\grid) = 0$, and so it must be that there is an element $\y\in\GCtb(\grid)$ with $\bdtoxb(\y) = \xp$. From the definition of the differential, we see that $\bdtoxb(\y) - \bdtox(\y) \in v \GCtb(\grid)$ so it must be that $\bdtox(\y) = \xp$. Thus, $\lambdatp(\grid) =0$, i.e.\ $\lambdahp(\Leg) =0$. Similarly, $\lhmb(\Leg) = 0$ implies $\lambdahm(\Leg) =0$.
\end{proof}

While we do not know if the converse of \fullref{thm:compare} is true, we have yet to find a counterexample. 
We have computed $\lambdahpm$ and $\lhpmb$ for almost all Legendrian representatives with  maximal Thurston-Bennequin number for knots of arc index up to $11$.\footnote{The majority of the computations were carried out by the second author, as part of his senior thesis.} 
In all examples we computed, the double-point enhanced GRID invariants were nonzero if and only if the (non-enhanced) GRID invariants were nonzero.

\bibliographystyle{mwamsalphack}
\bibliography{references}

\end{document}